\documentclass[12pt]{amsart}
\usepackage{amsmath}
\usepackage{amsfonts}
\usepackage{amssymb}

\newcommand\eps{\varepsilon}
\newcommand\R{{\mathbb{R}}}

\renewcommand\S{{\mathbf{S}}}

\theoremstyle{plain}
  \newtheorem{theorem}[subsection]{Theorem}
  
  \newtheorem{proposition}[subsection]{Proposition}
  \newtheorem{lemma}[subsection]{Lemma}

\theoremstyle{remark}
  \newtheorem{remark}[subsection]{Remark}

\theoremstyle{definition}
  \newtheorem{definition}[subsection]{Definition}

\begin{document}

\title[]{On existence of extremizers for the Tomas-Stein inequality for $S^1$}
\author{Shuanglin Shao}
\address{Department of Mathematics, University of Kansas, Lawrence, KS 66045}
\email{slshao@ku.edu}

\vspace{-0.1in}
\begin{abstract}The Tomas-Stein inequality or the adjoint Fourier restriction inequality for the sphere $S^1$ states that the mapping $f\mapsto \widehat{f\sigma}$ is bounded from $L^2(S^1)$ to $L^6(\R^2)$. We prove that there exists an extremizer for this inequality. We also prove that any extremizer satisfies $|f(-x)|=|f(x)|$ for almost every $x\in S^1$.
\end{abstract}

\maketitle

\section{Introduction}
The Tomas-Stein inequality or the adjoint Fourier restriction inequality for the sphere $S^1$ asserts that
\begin{equation}\label{eq-1}
\|\widehat{f\sigma}\|_{L^6(\R^2)} \le \mathcal{R}\|f\|_{L^2(S^1,\sigma)}
\end{equation} where the constant $\mathcal{R}>0$ is defined to be the optimal constant
\begin{equation}\label{eq-2}
 \mathcal{R}:=\sup\{\|\widehat{f\sigma}\|_{L^6(\R^2)}:\|f\|_{L^2(S^1,\sigma)}=
1\},
\end{equation} and $\sigma$ denotes the surface measure on the unit sphere $S^1$, and the Fourier transform is defined by
\begin{equation}\label{eq-3}
\widehat{f}(\xi):=\int e^{-i\xi\cdot x}f(x)dx.
\end{equation}
\begin{definition}
A function $f\in L^2(S^1)$ is said to be an extremizer or an extremal for \eqref{eq-1} if $f\neq 0$ a. e., and
\begin{equation}\label{eq-4}
\|\widehat{f\sigma}\|_{L^6(\R^2)} = \mathcal{R}\|f\|_{L^2(S^1)}.
\end{equation}
An extremizing sequence for the inequality \eqref{eq-1} is a sequence $\{f_\nu\}\in L^2(S^1)$ satisfies  $\|f_\nu\|_{L^2(S^1)}=1$ and $\lim_{\nu\to\infty }\|\widehat{f_\nu\sigma}\|_{L^6(\R^2)} =\mathcal{R}$.
An extremizing sequence is said to be precompact if any subsequence has a sub-subsequence which is Cauchy in $L^2(S^1)$.
\end{definition}
This paper is devoted to studying the existence of extremals for this basic inequality and to characterizing some properties of extremizers. The main result is the following 

\begin{theorem}\label{thm-main} There exists an extremal function for \eqref{eq-1}.  
\end{theorem}

Moreover we show that the extremizers enjoy the following symmetry.
\begin{theorem}\label{thm-even}
Every extremizer satisfies $|f(-x)|=|f(x)|$ for almost every $x\in S^1$.
\end{theorem}

In \cite{Christ-Shao:extremal-for-sphere-restriction-I-existence}, for the adjoint Fourier restriction inequality for the sphere $S^2$, we prove the existence of extremals by showing that any extremizing sequence of nonnegative functions is precompact in $L^2(S^2)$; and the extremals satisfy $|f(-x)|=|f(x)|$ for almost every $x\in S^2$. We are also able to prove that constants are local extremals.  Furthermore in \cite{Christ-Shao:extremal-for-sphere-restriction-II-characterizations}, we show that nonnegative extremizers are indeed smooth, and completely characterize complex extremals: Any complex extremizer is of the form $e^{ix\xi}f(x)$ for some nonnegative extremizer $f$ and some $\xi\in \mathbb{R}^3$, and if $\{f_\nu\}$ is a complex extremizing sequence then there exists $\{\xi_\nu\}$ such that $\{e^{-ix\cdot\xi_\nu}f_\nu\}$ is precompact. Recently, in \cite{Foschi:2015:maximizers-2D-sphere}, Foschi proves that constant functions are extremizers for the two dimensional sphere. In \cite{Fanelli-Vega-Visciglia:2011:maximizers-restriction}, Fanelli, Vega and Visciglia consider similar questions and establish existence for a family of non-endpoint Fourier restriction operators. 

In the context of the adjoint Fourier restriction inequality for the paraboloid (or the Strichartz inequality for the Schr\"odinger equation), Kunze \cite{Kunze:2003:maxi-strichartz-1d} proves the existence of extremals when the spatial dimension is one by a concentration-compactness argument. Foschi \cite{Foschi:2007:maxi-strichartz-2d} proves that Gaussian functions are explicit extremals in spatial dimensions one and two by two successive applications of the Cauchy-Schwarz inequalities; independently Hundertmark and Zharnitsky \cite{Hundertmark-Zharnitsky:2006:maximizers-Strichartz-low-dimensions} obtain similar results. Bennett, Bez, Carbery and Hundertmark \cite{Bennett-Bez-Carbery-Hundertmark:2008:heat-flow-of-strichartz-norm} show that Gaussians are extremizers from the perspective of the heat-flow deformation method \cite{Bennett-Carbery-Christ-Tao-2008-Brascamp-Lieb-inequality, Bennett-Carbery-Tao:2006:multilinear-restri-kakeya}. For non-$L^2$ adjoint Fourier restriction inequality for paraboloids, Christ and Quilodr\'an \cite{Christ-Quilodran:Gaussians-rarely-extremize-non-L2-restriction-paraboloids} show that Gaussians are rarely extremizers by studying the corresponding Euler-Lagrange equations. In higher dimensions, the existence result of extremizers is known, which is achieved by using the tool of profile decompositions by the author \cite{Shao:2008:maximizers-Strichartz-Sobolev-Strichartz}. In the context of a convolution inequality with the surface measures of the paraboloids, the existence of quasi-extremals and extremals was studied by Christ \cite{Christ:extremizer-radon-like-transform, Christ:quasiextremizer-radon-like-transform}.

The main results Theorem \ref{thm-main} and Theorem \ref{thm-even} are proven by following the framework designed in the paper \cite{Christ-Shao:extremal-for-sphere-restriction-I-existence}. To be more precise, the first part of the analysis in this paper, Step 1 to Step 4, follows similarly as in \cite{Christ-Shao:extremal-for-sphere-restriction-I-existence} to obtain a nonnegative extremizing sequence $f_\nu$ which is ``even upper normalized" with respect to a sequence of caps $\mathcal{C}_\nu$. In the second part of the analysis, i.e., at the last Step 5, we develop a profile decomposition for the adjoint sphere restriction operator in the spirit of \cite{Bahouri-Gerard:1999:profile-wave, Carles-Keraani:2007:profile-schrod-1d}. When $f$ is supported on sufficiently small caps on the sphere, we approximate $\widehat{f\sigma}$ by linear Schr\"odinger waves. This idea appeared previously in \cite{Christ-Colliander-Tao:2003:asymptotics-modulation-canonical-defocusing-eqs} where the authors approximate the Airy wave at high frequency by a Schr\"odinger wave.   

The analysis in this paper can be viewed as a manifestation of the the concentration-compactness approach developed in a series of works by Lions \cite{Lions-1984-cc-locally-compact-I, Lions-1984-cc-locally-compact-II, Lions-1985-cc-limit-case-I, Lions-1985-cc-limit-case-II}, which is however adapted to our case to cope with the nonlocal characteristics of the adjoint Fourier restriction operator.

We present the outline and results in detail in Section \ref{sec:notation-outline}.

\textbf{Acknowledgements.}  The research of the author was supported by NSF DMS-1160981.

\section{Outline of the proof and definitions}\label{sec:notation-outline}
This section consists of notations, definitions and statements of some intermediate results which are not repeated subsequently. We start with several definitions. Let $\sigma_P$ be the canonical measure on the parabola $P=\{(x,y)\in \R^2:y=\frac 12 |x|^2\}$, and set
\begin{align}
\label{eq-75}\mathcal{R}_P&:=\sup\frac {\|\widehat{f\sigma_P}\|_{L^6(\R^2)}}{\|f\|_{L^2(P, \sigma_P)}},\\
\mathbf{S}&:=\sup \frac {\|f\sigma\ast f\sigma\ast f\sigma\|^{1/3}_{L^2(\R^3)}}{\|f\|_{L^2(S^1, \sigma)}}\\
\mathbf{P}&:=\sup \frac {\|f\sigma_P \ast f\sigma_P \ast f\sigma_P\|^{1/3}_{L^2(\R^3)}}{\|f\|_{L^2(P, \sigma_P)}}.
\end{align} Note that by Plancherel's theorem, $\mathcal{R}=2\pi \S, \, \mathcal{R}_P =2\pi \mathbf{P}.$ There holds that $|f\sigma * f\sigma *f\sigma | \le |f|\sigma* |f|\sigma *|f|\sigma$. If $f$ is an extremizer to \eqref{eq-1}, so if $|f|$. This applies to any extremizing sequence $\{f_\nu\}$. Thus in order to prove the existence of extermizers, we will restrict our attention to nonnegative functions and nonnegative extremizing sequences. 

\textbf{Step 1, A strict comparison.}  By a dilation argument, we see that the sharp constants for the adjoint Fourier restriction inequalities for the sphere and the paraboloid satisfy, $\mathcal{R}\ge \mathcal{R}_P$, where $\mathcal{R}_P$ is defined in \eqref{eq-75}. This reasoning appears previously in \cite{Christ-Shao:extremal-for-sphere-restriction-I-existence}.  Indeed, we take an extremizer for the paraboloid, which is known as Gaussians from  \cite{Foschi:2007:maxi-strichartz-2d} and \cite{Hundertmark-Zharnitsky:2006:maximizers-Strichartz-low-dimensions}, dilate it so that it is essentially supported in a sufficiently small set of the paraboloid; we paste the extremizer onto the sphere in an obvious way and then osculate the sphere by the parabolic scaling $(x',x_d)\to (\lambda x',\lambda^2 x_d)$ where $x=(x',x_d)\in \R^d$ and $\lambda>0$. In the limits, we see that the relation $\mathcal{R}\ge \mathcal{R}_P$ holds. So there arises the most severe obstruction to the existence of extremizers that, for an extremizing sequence $\{f_\nu\}$ satisfying $\|f_\nu\|_2=1$, any subsequential weak limit of $|f_\nu|^2$ could conceivably converge to a Dirac mass at a point of $S^1$. If it were the case, then $\mathcal{R}=\mathcal{R}_P$. To rule out this scenery, an essential step is to prove $\mathcal{R}>\mathcal{R}_P$. Because any extremal enjoys a symmetry $|f(x)|=|f(-x)|$, there is a possibility that the extremizing sequence might converge weakly to a linear combination of two Dirac masses at antipodal points of $S^1$. To rule it out, one needs a strict comparison $\mathbf{S}>(5/2)^{1/6}\mathbf{P}$, which is achieved by using a perturbation argument, which we sketch in Appendix \ref{sec:strict-comparison}.
\begin{proposition}\label{prop-strict-inequ}
\begin{equation}
\mathcal{R}>(5/2)^{1/6}\mathcal{R}_P.
\end{equation}
\end{proposition}

\textbf{Step 2, Antipodal symmetrization.} We will show that ``extremals" to \eqref{eq-1} enjoy a symmetry $|f(-x)|=|f(x)|$.
\begin{definition}\label{def-even}
A complex-valued function $f\in L^2(S^1)$ is said to be even if $\overline{f(-x)}=f(x)$ for almost every $x\in S^1$. For nonnegative functions, this condition is simplified to $f(-x)=f(x)$.
\end{definition}

\begin{definition} Let $f$ be nonnegative $L^2(S^1)$ function. The antipodally symmetric rearrangement $f_\star$ is the unique non-negative element of $L^2(S^1)$ which satisfies
\begin{align}
f_\star(x)&=f_\star(-x),\text{ for all }x\in S^1,\\
f_\star(x)^2+f_\star(-x)^2&= f(x)^2+f(-x)^2, \text{ for all }x\in S^1.
\end{align}
In other words, $f_\star=\sqrt{\frac {f(x)^2+f(-x)^2}{2}}$ and $\|f_\star\|_2=\|f\|_2$.
\end{definition}
\begin{proposition}\label{prop-even-arrangement}For any nonnegative function $f\in L^2(S^1)$,
\begin{equation}\label{eq-5}
\|f\sigma\ast f\sigma\ast f\sigma\|_2 \le \|f_\star\sigma\ast f_\star \sigma\ast f_\star\sigma\|_2,
\end{equation}
with strict inequality if and only if $f=f_\star$ almost everywhere. Consequently any extremizer for the inequality \eqref{eq-1} satisfies $|f(-x)|=|f(x)|$ for almost every $x\in S^1$.
\end{proposition}

The analogue for $S^2$ is establish in \cite{Christ-Shao:extremal-for-sphere-restriction-I-existence}. We remark that Foschi \cite{Foschi:2015:maximizers-2D-sphere} has provided a much shorter proof for that by using the Cauchy-Schwarz inequality. 

\textbf{Step 3, A refinement of Tomas-Stein's inequality.}
Similarly as in \cite{Christ-Shao:extremal-for-sphere-restriction-I-existence}, we define what caps mean on $S^1$. 
\begin{definition}\label{def-cap}
The \emph{cap} $\mathcal{C}=\mathcal{C}(z,r)$ with center $z\in S^1$ and radius $r\in (0,1]$ is the set of all points $y\in S^1$ which lie in the same hemisphere as $z$ and are centered at $z$, and which satisfy $\pi_{H_z}(y)<r$, where the subspace $H_z\subset \R^2$ is the orthogonal complement of $z$ and $\pi_{H_z}$ denotes the orthogonal projection onto $H_z$.
\end{definition}

 In \cite{Christ-Shao:extremal-for-sphere-restriction-I-existence}, the refinement of Tomas-Stein's inequality for $S^2$ developed by  Bourgain \cite{Bourgain:1991:restri-multiplier} and Moyua, Vargas and Vega \cite{Moyua-Vargas-Vega:1999} provides some useful information on the near-extremals for the adjoint Fourier restriction inequality for $S^2$: Any near extremal can be decomposed into a major part, which obey some upper bound in the point-wise sense, and a lower $L^2$-norm bound, plus an error term. For $S^1$, we have the following refinement
\begin{lemma}\label{le-refinement-of-Tomas-Stein} For $f\in L^2(S^1)$. There exists $\alpha\in (0,1)$ such that
\begin{equation}\label{eq-r1}
\|\widehat{f\sigma}\|_6 \le \left(\sup_\mathcal{C}\frac {1}{|\mathcal{C}|^{1/2}} \int_\mathcal{C} |f|d\sigma \right)^{\alpha} \|f\|^{1-\alpha}_{L^2(S^1)},
\end{equation} where $\mathcal{C}$ denotes a cap on $S^1$.
\end{lemma}
We establish this lemma by using the bilinear restriction estimates for functions on $S^1$ whose supports are ``transverse", i.e., the unit normals to each set are separated by an angle $>0$. The argument is similar to that for \cite[Theroem 1.3]{Begout-Vargas:2007:profile-schrod-higher-d}.

As a consequence of the refinement in Lemma \ref{le-refinement-of-Tomas-Stein}, we have
\begin{proposition}\label{prop-refinement-appliedto-near-extremal}
For any $\delta>0$ there exists $C_\delta<\infty$ and $\eta_\delta>0$ with the following properties. If $f\in L^2(S^1)$ satisfies $\|\widehat{f\sigma}\|_6 \ge \delta \mathcal{R} \|f\|_2$, then there exists a decomposition $f=g+h$ and a cap $\mathcal{C}$ satisfying that
\begin{align}
& 0\le |g|,|h|\le |f|,\\
& g,h \text{ have disjoint supports},\\
& |g(x)|\le C_\delta \|f\|_2|\mathcal{C}|^{-1/2} \chi_\mathcal{C}(x), \forall \,x,\\
& \|g\|_2\ge \eta_\delta \|f\|_2.
\end{align} Here both $C^{-1}_\delta$ and $\eta_\delta$ are proportional to $\delta^{O(1)}$.
If $f\ge 0$, $g$ and $h$ can be chosen such that $g,h\ge 0$ almost everywhere.
\end{proposition}

\textbf{Step 4, Upper even normalized w.r.s.t the caps.} As in \cite{Christ-Shao:extremal-for-sphere-restriction-I-existence}, we introduce the notion of rescaling maps that pull back functions on $S^1$ to $\R$ and obtain some preliminary control on the near-extremals.
\begin{definition}[Rescaling map $\phi_\mathcal{C}$]
Let $\mathcal{B}\subset \R$ denote the unit ball. To any cap of radius $\le 1$ is associated a rescaling map $\phi_\mathcal{C}: \mathcal{B}\to \mathcal{C}$. For $z=(0,1)$, $\phi_\mathcal{C}(y)=(ry,\sqrt{1-r^2y^2})$. For general $z$, define $\psi_z(y)=r^{-1}L(\pi(y))$ where $\pi$ is again the orthogonal projection onto $H_z$, $L:H_z\to \R$ is an arbitrary linear isometry and $\phi_\mathcal{C}=\psi_z^{-1}$.
\end{definition}
For a cap $\mathcal{C}= \mathcal{C}(z,r)$, we remark that $\phi_\mathcal{C}$ is naturally extended to defined on the set $\{y:\, |y|<r\}$.

\begin{definition}[Pullbacks]
Define the pullbacks $\phi_\mathcal{C}^*=r^{1/2}(f\circ\phi_\mathcal{C})(y)$ where $r$ is the radius of the cap $\mathcal{C}$.
\end{definition}
\begin{remark} This definition of pullbacks makes sense if $f$ is supported in the cap of radius $1$ concentric with $\mathcal{C}$.
\end{remark}

\begin{definition}[Upper normalized w.r.t. caps and balls]
Let $\Theta:[1,\infty)\to (0,\infty)$ satisfy $\Theta(R)\to 0$ as $R\to \infty$. A function $f\in L^2(S^1)$ is said to be upper normalized with respect to a cap $\mathcal{C}=\mathcal{C}(z,r)\subset S^1$ of radius $r$ and center $z$ if the following hold
\begin{align}
& \|f\|_2\le C<\infty,\\
& \int_{|f|>Rr^{-1/2}}|f|^2d\sigma(x) \le \Theta(R), \forall \,R\ge 1,\\
& \int_{|x-z|\ge Rr} |f|^2d\sigma(x) \le \Theta(R), \forall \,R\ge 1.
\end{align}An even function $f$ is said to be upper even-normalized with respect to $\Theta$, and $\mathcal{C}$ if $f$ can be decomposed into $f=f_++f_-$ where $f_-(x)=\overline{f_+(-x)}$, and $f_+$ is upper normalized with respect to $\Theta$ and $\mathcal{C}$.
A function $f\in L^2(\R)$ is said to be upper normalized with respect to the unit ball in $\R$ if
\begin{align}
& \|f\|_2\le C<\infty,\\
& \int_{|f|>R}|f|^2dx \le \Theta(R), \forall \,R\ge 1,\\
& \int_{|x|\ge R} |f|^2dx \le \Theta(R), \forall \,R\ge 1.
\end{align}
\end{definition}
\begin{definition}[Near-extremal]
A nonzero function $f\in L^2(S^1)$ is said to be $\delta$-nearly extremal for the inequality \eqref{eq-1} if
\begin{equation}\label{eq-35}
\|f\sigma \ast f\sigma \ast f\sigma \|_2\ge (1-\delta)^3 \mathbf{S}^3\|f\|_2^3.
\end{equation}
\end{definition}
The following proposition provides a preliminary decomposition for nearly extremals.
\begin{proposition}\label{prop-prelimi-decomp-near-extremals}
There exists a function $\Theta:\,[1,\infty)\to (0,\infty)$ satisfying $\Theta(R)\to 0$ as $R\to \infty$ with the following property. For any $\eps>0$, there exists $\delta >0$ such that any nonnegative even functions $f\in L^2(S^1) $ satisfying $\|f\|_2=1$ which is a $\delta$-nearly extremal may be decomposed as $f=F+G$, where $F$ and $G$ are even and nonnegative with disjoint supports, $\|G\|_2\le \eps$ and there exists a cap $\mathcal{C}$ such that $F$ is upper even-normalized with respect to $\mathcal{C}$.
\end{proposition}

It follows from two crucial facts: the refinement of Tomas-Stein's inequality for $S^1$ in Proposition \ref{prop-refinement-appliedto-near-extremal}, and a geometric fact that ``distant caps interact weakly" we will establish in Section \ref{sec:refined-tomas-stein}. The latter asserts, roughly speaking, that $\|\chi_\mathcal{C}\sigma \ast \chi_\mathcal{C}\sigma \ast  \chi_\mathcal{C'}\sigma \|_{2}\ll |\mathcal{C}||\mathcal{C'}|^{1/2}$ unless the caps $\mathcal{C}$ and $\mathcal{C'}$ have comparable radii and nearby centers.

\textbf{Step 5, Ruling out small caps and existence of extremals.} In \cite{Christ-Colliander-Tao:2003:asymptotics-modulation-canonical-defocusing-eqs}, the authors observe that the linear Airy evolution at high frequency is well approximated by a linear Schr\"odigner evolution, which is used in \cite{Shao:2008:linear-profile-Airy-Maximizer-Airy-Strichartz} to establish the linear profile decomposition for the Airy equation. In this paper, we have observed that a similar phenomena occurs for $\widehat{f\sigma}$ when $f$ is supported on a very small cap. More precisely,  given any extremizing sequence $\{f_\nu\}$ which is upper even normalized with respect to caps $C_\nu$ with radii $r_\nu\to 0$, $\widehat{f_\nu\sigma}$ can be written as a superposition of ``orthogonal" linear Schr\"odinger waves, plus a small error term. In this case, there follows that $$ \mathcal{R} \le (5/2)^{1/6}\mathcal{R}_P.$$ But it is a contradiction to the strict inequality that $ \mathcal{R} >(5/2)^{1/6}\mathcal{R}_P$. Thus $\inf_\nu r_\nu>0$. Then for ``large caps", one can indeed prove $f_\nu$ is precompact, which leads to the existence of extremals for \eqref{eq-1}.
\begin{proposition}\label{prop-regularity-after-rescaling}
Let $\{f_\nu\}\subset L^2(S^1)$ be an extremizing sequence for the inequality \eqref{eq-1} satisfying $\|f_\nu\|_2=1$ and $|f(-x)|=|f(x)|$ for a.e. $x\in S^1$. Suppose that each $|f_\nu|$ is upper even-normalized with respect to a cap $\mathcal{C}_\nu\cup (-\mathcal{C}_\nu)$ where $\mathcal{C}_\nu=\mathcal{C}(z_\nu,r_\nu)$, with constants uniform in $\nu$. Then $$ \inf_\nu r_\nu > 0.$$
In this case, an extremal for \eqref{eq-1} is obtained. 
\end{proposition}

The strict comparison on the optimal constants for Tomas-Stein's inequalities for the sphere and the paraboloid is essential to obtain existence of extremals to Tomas-Stein's inequality for the sphere. In high dimensions, it is the lack of the strict comparison on the optimal constants and the algebraic property of the even integer $6$ that prevent us from obtaining the existence of extremals.

Step 1 is established in the Appendix. We present Step 2 through Step 5 in what follows.

\section{Step 2. Antipodal symmetrization}\label{sec:symmetrization}
In this section, we will prove the functional $\|f\sigma \ast f\sigma \ast f\sigma \|_2^2/\|f\|_2^6$ is non-decreasing under the antipodal symmetrization defined in Definition \ref{def-even}.
\begin{proof}[Proof of Proposition \ref{prop-even-arrangement}]
For $f\ge 0$,
\begin{equation}\label{eq-7}
\|f\sigma \ast f\sigma \ast f\sigma \|_2^2 =\int f(a_1)\times \cdots \times f(a_6) d\lambda(a_1,\cdots, a_6)
\end{equation} for a certain non-negative measure $\lambda$ which is supported by the set
 \begin{equation}\label{eq-6}
\{(a_1,\cdots, a_6)\in (\R^2)^6:\,a_1+a_2+a_3=a_4+a_5+a_6\},
\end{equation} and which is invariant under the following transformations
\begin{equation}
\begin{split}
(a_1,a_2,a_3,a_4,a_5,a_6) &\mapsto (a_4,a_5,a_6,a_1,a_2,a_3),\\
(a_1,a_2,a_3,a_4,a_5,a_6)&\mapsto \bigl(a_{\tau(1)},a_{\tau(2)},a_{\tau(3)},a_4,a_5,a_6\bigr),\\
(a_1,a_2,a_3,a_4,a_5,a_6) &\mapsto (a_1,a_2,-a_4,-a_3,a_5,a_6),\\
(a_1,a_2,a_3,a_4,a_5,a_6)&\mapsto (a_1,-a_4,-a_5,-a_2,-a_3,a_6),
\end{split}
\end{equation}where $\tau \in S^3$, the permutation group of order 3. We denote by $G$ the finite group of symmetries of $(\R^2)^6$ generated by these symmetries. The cardinality of $G$ is $2\times 6!$ since there holds a short exact sequence
\begin{equation}
1\mapsto \{\pm 1\}\mapsto G\mapsto S^6\mapsto 1.
\end{equation} Note that in order for a sequence $(a_1,a_2,a_3,a_4,a_5,a_6)$ of fixed order to satisfy the requirement \eqref{eq-6}, the only way is to add ``$-$" sign. Hence from basic algebra, there holds that $|G/\{\pm 1\}|=|S^6|=6!$; thus $|G|=2\times 6!$ follows.

By the orbit of a point we mean its image under $G$; by a generic point we mean one whose orbit has cardinality $2\times 6!$. In \eqref{eq-7}, it suffices to integrate only over all generic $6$-tuples $(a_1,\cdots, a_6)$ satisfying \eqref{eq-6}, since they form a set of full $\lambda$-measure.

To the orbit $\mathcal{O}$ we associate the functions
\begin{equation}
\begin{split}
\mathcal{F}(\mathcal{O})&=\sum_{(a_1,\cdots, a_6)\in \mathcal{O}} f(a_1)\times \cdots \times f(a_6),\\
\mathcal{F}_\star(\mathcal{O})&=\sum_{(a_1,\cdots, a_6)\in \mathcal{O}} f_\star(a_1)\times \cdots \times f_\star(a_6)
\end{split}
\end{equation}
Let $\Omega$ denote the set of all orbits of generic points. We can write
\begin{equation}
\begin{split}
\|f\ast f\ast f\|_2^2 &=\int_{\Omega} \mathcal{F}(\mathcal{O})d\mu(\mathcal{O}),\\
\|f_\star \ast f_\star \ast f_\star\|_2^2&=\int_{\Omega} \mathcal{F}_\star(\mathcal{O})d\mu(\mathcal{O})
\end{split}
\end{equation} for a certain nonnegative measure $\mu$. Therefor it suffices to prove that for any generic orbit $\mathcal{O}$,
\begin{equation}\label{eq-8}
\sum_{(a_1,\cdots, a_6)\in \mathcal{O}} f(a_1)\times \cdots \times f(a_6) \le \sum_{(a_1,\cdots, a_6)\in \mathcal{O}} f_\star(a_1)\times \cdots \times f_\star(a_6).
\end{equation}
Fix any generic orbit $6$-tuple $(a_1,\cdots, a_6)$ satisfying \eqref{eq-6}, we prove \eqref{eq-8} for its orbit. By homogeneity, we may assume that $f(a_1)^2+f(-a_1)^2=1$ and that the same holds simultaneously for $a_i$ for $i=2,\cdots, 6$. Thus we may write
\begin{equation}
\begin{split}
&f(a_1)=\cos (\theta_1), f(a_2)=\cos(\theta_2), \cdots, f(a_6)=\cos(\theta_6),\\
&f(-a_1)=\sin (\theta_1), f(-a_2)=\sin(\theta_2), \cdots, f(-a_6)=\sin(\theta_6)
\end{split}
\end{equation} for $\theta_i\in [0, \frac \pi 2]$ for $i=1,\ldots, 6$. Thus by definition
\begin{equation}
f_\star =2^{-1/2}.
\end{equation}
Before writing out $\sum_{(a_1,\cdots, a_6)\in \mathcal{O}} f(a_1)\times \cdots \times f(a_6)$ for a generic orbit of $(a_1^\prime,\cdots, a_6^\prime)$, we note that there will be $2\times 6!$ summands, which can be organized into $20$ terms. Here $20$ comes out because $20=\frac {2\times 6!}{2\times 3!\times 3!}$. Below we will write out a long formula,
\begin{equation}
\begin{split}
&\frac {1}{2\times 3!\times 3!}\sum_{(a_1,\cdots, a_6)\in \mathcal{O}} f(a_1)\times \cdots \times f(a_6)\\
&=\cos(\theta_1)\cos(\theta_2)\cos(\theta_3)\cos(\theta_4)\cos(\theta_5)\cos(\theta_6)\\
&+\sin(\theta_1)\sin(\theta_2)\sin(\theta_3)\sin(\theta_4)\sin(\theta_5)\sin(\theta_6)\\
&+\begin{cases}\cos(\theta_1)\cos(\theta_2)\sin(\theta_3)\\
\cos(\theta_1)\sin(\theta_2)\cos(\theta_3)\\
\sin(\theta_1)\cos(\theta_2)\cos(\theta_3)
\end{cases}\times \begin{cases}\sin(\theta_4)
\cos(\theta_5)\cos(\theta_6)\\
\cos(\theta_4)\sin(\theta_5)\cos(\theta_6)\\
\cos(\theta_4)\cos(\theta_5)\sin(\theta_6)
\end{cases}\\
&+\begin{cases}\cos(\theta_1)\sin(\theta_2)\sin(\theta_3)\\
\sin(\theta_1)\cos(\theta_2)\sin(\theta_3)\\
\sin(\theta_1)\sin(\theta_2)\cos(\theta_3)\\
\end{cases}\times \begin{cases}
\sin(\theta_4)\sin(\theta_5)\cos(\theta_6)\\
\sin(\theta_4)\cos(\theta_5)\sin(\theta_6)\\
\cos(\theta_4)\sin(\theta_5)\sin(\theta_6)\\
\end{cases}\\
&=:\Gamma(\theta_1,\cdots, \theta_6).
\end{split}
\end{equation}
We will organize terms according to $\cos(\theta_5)\cos(\theta_6)$, $\sin(\theta_5)\sin(\theta_6)$, $\sin(\theta_5)\cos(\theta_6)$ and $\cos(\theta_5)\sin(\theta_6)$ and rewrite $\Gamma$ as
\begin{equation}\label{eq-9}
\begin{split}
\Gamma &=\cos(\theta_5)\cos(\theta_6) A+ \sin(\theta_5)\sin(\theta_6)B+ \bigl(\sin(\theta_5)\cos(\theta_6)+\cos(\theta_5)\sin(\theta_6)\bigr)C\\
&=\cos(\theta_5)\cos(\theta_6) A+ \sin(\theta_5)\sin(\theta_6)B+ \sin(\theta_5+\theta_6)C
\end{split}
\end{equation}where $A$ and $B$ contain $4$ terms, respectively, and $C$ contains $6$ terms. For the sake of simplicity, we will not explicitly write out $A$, $B$ and $C$; we will do so whenever it is necessary.

We aim to show that \begin{equation}\label{eq-10}
\max_{(\theta_1,\cdots, \theta_6)\in [0,\pi/2]^6}\Gamma=\Gamma (\pi/4,\cdots, \pi/4)=\frac {5}{2}.
\end{equation} This maximal value of $\Gamma$ matches values taken by $$\frac {1}{2\times 3!\times 3!}\sum_{(a_1,\cdots, a_6)} f_\star (a_1)\times \cdots\times f_\star(a_6)=\frac {20}{8}=\frac 52. $$

Suppose that $(\alpha_1,\cdots,\alpha_6)$ is a critical point of $\Gamma$, then at this point, there holds that
\begin{equation}
\frac {d\Gamma}{d\theta_5}=\frac {d\Gamma}{d\theta_6}=0,
\end{equation}which implies that, at this critical point,
\begin{equation}\label{eq-11}
0=\frac {d\Gamma}{d\theta_5}-\frac {d\Gamma}{d\theta_6}.
\end{equation}
We expand the right hand side of \eqref{eq-11} out to see
\begin{equation}\label{eq-12}
\sin(\alpha_5-\alpha_6)(A+B)=0.
\end{equation}
We will show that $A+B|_{(\alpha_1,\ldots,\alpha_6)}>0$; otherwise, since every term in $A+B$ is nonnegative, that $A+B=0$ will imply that every term is actually zero. This further implies that
\begin{equation}\label{eq-13}
\Gamma (\alpha_1,\cdots, \alpha_6)< \frac {5}{2},
\end{equation} which contradicts to the definition that $(\alpha_1,\cdots,\alpha_6)$ is a critical point.
Indeed, We writes out
\begin{equation}
\begin{split}
A&=\cos(\theta_1)\cos(\theta_2)\cos(\theta_3-\theta_4)+\sin(\theta_1+\theta_2)\cos(\theta_3)\sin(\theta_4),\\
B&=\sin(\theta_1)\sin(\theta_2)\cos(\theta_3-\theta_4)+\sin(\theta_1+\theta_2)\sin(\theta_3)\cos(\theta_4).
\end{split}
\end{equation}
Hence \begin{equation}
0=A+B=\cos(\theta_1-\theta_2)\cos(\theta_3-\theta_4)+\sin(\theta_1+\theta_2)\sin(\theta_3+\theta_4).
\end{equation}This implies that

\begin{equation}
\begin{cases}
\cos(\theta_1-\theta_2)\cos(\theta_3-\theta_4)=0,\\
\sin(\theta_1+\theta_2)\sin(\theta_3+\theta_4)=0.
\end{cases}
\end{equation}So we have the following four combinations,
\begin{equation}
\begin{cases}
\theta_1-\theta_2=\pi/2,\\
\theta_3+\theta_4=0,\text{ or }\pi.\\
\end{cases}
\begin{cases}
\theta_1-\theta_2=-\pi/2,\\
\theta_3+\theta_4=0,\text{ or }\pi.\\
\end{cases}
\begin{cases}
\theta_1+\theta_2=0,\text{ or }\pi,\\
\theta_3-\theta_4=-\pi/2.\\
\end{cases}
\begin{cases}
\theta_1+\theta_2=0,\text{ or }\pi,\\
\theta_3-\theta_4=\pi/2.\\
\end{cases}
\end{equation}
In all cases, we can show that $$\Gamma <\frac 52.$$ For instance, assume the first instance with
$$\theta_1-\theta_2 =\pi/2,\,\theta_3+\theta_4=0.$$
Then by the fact that all $\theta_i\in [0,\pi/2]$, we have
$$\theta_1=\pi/2, \theta_2=\theta_3=\theta_4=0.$$
Then $\Gamma$ is  simplified to
$$\Gamma=\sin(\theta_5+\theta_6)\le 1<5/2.$$

Thus \eqref{eq-12} forces that
\begin{equation}
\theta_5=\theta_6.
\end{equation}
By symmetry it immediately follows that
\begin{equation}\label{eq-14}
\begin{split}
\theta_4=\theta_5=\theta_6:=\beta,\\
\theta_1=\theta_2=\theta_3:=\alpha
\end{split}
\end{equation}
Observing that by symmetry we may exchange, say, $f(a_3)=\cos(\theta_3)=\cos(\alpha)$ and $f(-a_4)=\sin(\theta_4)=\sin(\beta)=\cos(\frac \pi2-\beta)$ in $\Gamma$, we conclude that
\begin{equation}\label{eq-15}
\alpha+\beta=\pi/2.
\end{equation}
Combining \eqref{eq-14} and \eqref{eq-15}, we see that
\begin{equation}
\Gamma (\alpha_1,\cdots,\alpha_6)=20\cos^3(\alpha)\cos^3(\beta)=\frac {20}{8}\bigl(2\sin(\alpha)\cos(\alpha)\bigr)^3=\frac {5}{2}\sin^3(2\alpha)\le \frac {5}{2}
\end{equation} with ``=" if and only if $\alpha=\frac \pi 4$. Hence the only choice for critical points is
\begin{equation}\alpha=\beta=\frac \pi 4.\end{equation} To conclude, we have established the claim \eqref{eq-10}. Hence \eqref{eq-8} follows. Therefore the proof of Theorem \eqref{prop-even-arrangement} is complete.
\end{proof}

\section{Step 3. A refined estimate and a geometric fact}\label{sec:refined-tomas-stein}
In this section we first establish the refinement of Tomas-Stein inequality for $S^1$ in Lemma \ref{le-refinement-of-Tomas-Stein}, which easily implies Proposition \ref{prop-refinement-appliedto-near-extremal}, see e.g. \cite{Carles-Keraani:2007:profile-schrod-1d,Shao:2008:linear-profile-Airy-Maximizer-Airy-Strichartz}. In the end of this section, we establish a geometric fact that ``distant caps interact weakly".

We aim to prove the following
\begin{proposition}\label{prop:refined-Tomas-Stein}
For $f\in L^2(S^1)$. There exists $\alpha\in (0,1)$ such that
\begin{equation*}
\|\widehat{f\sigma}\|_6 \le \left(\sup_\mathcal{C}\frac {1}{|\mathcal{C}|^{1/2}} \int_\mathcal{C} |f|d\sigma \right)^{\alpha} \|f\|^{1-\alpha}_{L^2(S^1)},
\end{equation*} where $\mathcal{C}$ denotes a cap on $S^1$.
\end{proposition}

We recall two lemmas: The first is on the bilinear restriction estimates for functions whose supports are transverse.
\begin{lemma}\label{le-bilinear}
Let $f, g\in L^2(S^1)$ and assume that $f$ and $g$ are supported on the caps $\mathcal{C}_1$ and $\mathcal{C}_2$, which are separated by $2^j$ for $j\le 0$. Then
\begin{equation}\label{eq-r2}
\|\widehat{f\sigma} \widehat{g\sigma}\|_2 \le C 2^{-j/2} \|f\|_2\|g\|_2.
\end{equation}
\end{lemma}
The proof of this lemma follows from Cauchy-Schwarz's inequality and the geometric estimate $\|\sigma\ast \sigma'\|_\infty \le C2^{-j}$ where $\sigma$ and $\sigma'$ denote the surface measures supported on the two caps on the sphere $S^1$ which are separated by an angle $\gtrsim 2^{j}$, $j\le 0$.

The second is on the ``almost orthogonality" for functions which have disjoint supports in the Fourier space, see for instance \cite[Lemma 6.1]{Tao-Vargas-Vega:1998:bilinear-restri-kakeya}.
\begin{lemma}\label{le-almost-ortho}
Let $\{R_k\}$ be a collection of rectangles and $c>0$ such that the dilates $(1 + c)R_k$ are almost disjoint (i.e.,
$\sum_k\chi_{(1+c)R_k}\le C<\infty$), and suppose that $\{f_k\}$ is a collection of functions supported by $\{R_k\}$. Then for all $1\le p\le \infty$,
\begin{equation}
\|\sum_k \widehat{f_k}\|_p \lesssim \left(\sum_k \|\widehat{f_k}\|_p^{p_0}\right)^{1/p_0}
\end{equation}where $p_0=\min\{p,p/(p-1)\}$.
\end{lemma}

\begin{proof}[Proof of Proposition \ref{prop:refined-Tomas-Stein}]
For $j\le 0$, we partition $S^1$ into a union of caps $\mathcal{C}^j_k$ of length $\sim 2^j$. For a given cap $\mathcal{C}^j_k$, we partition it into two equal caps $\mathcal{C}^{j-1}_{m}$ and $\mathcal{C}^{j-1}_{n}$ of length $2^{j-1}$; we say $\mathcal{C}^j_k$ is the ``parent" of $\mathcal{C}^{j-1}_{m}$ and $\mathcal{C}^{j-1}_{n}$. We define a relation $\mathcal{C}^j_{k_1}\sim \mathcal{C}^j_{k_2}$ if they are not adjoint but their parents are adjoint. Then
\begin{equation}\label{eq-r3}
\widehat{f\sigma}\widehat{f\sigma}=\sum_j \sum_{k}\sum_{l:\,\mathcal{C}^j_l\sim \mathcal{C}^j_k}  \widehat{f_k\sigma}\widehat{f_l\sigma},
\end{equation}
where $f_k:=f\chi_{\mathcal{C}^j_k}$ and $f_l:=f\chi_{\mathcal{C}^j_l}$. We write $\sum_j\sum_{(k,l)} :=\sum_j \sum_{k}\sum_{l:\,\mathcal{C}^j_l\sim \mathcal{C}^j_k}$.

By Lemma \ref{le-bilinear},
$$ \|\widehat{f_k\sigma} \widehat{f_l\sigma}\|_2 \le C2^{-j/2} \|f_k\|_2\|f_l\|_2.$$
On the other hand,
$$ \|\widehat{f_k\sigma} \widehat{f_l\sigma}\|_\infty \le C\|f_k\|_1\|f_l\|_1.$$
Then by interpolation,
\begin{equation}\label{eq-r4}
\|\widehat{f_k\sigma} \widehat{f_l\sigma}\|_3 \le C2^{-\frac j3} \|f_k\|_{3/2}\|f_l\|_{3/2}.
\end{equation}
From \eqref{eq-r3} and \eqref{eq-r4}, and Lemma \ref{le-almost-ortho},
\begin{equation}\label{eq-r6}
\begin{split}
\|\widehat{f\sigma}\|_6^2 &=\|\widehat{f\sigma}\widehat{f\sigma} \|_3=\|\sum_j\sum_{(k,l)}\widehat{f_k\sigma}\widehat{f_l\sigma} \|_3\\
&\lesssim \left( \sum_j\sum_{(k,l)} \|\widehat{f_k\sigma} \widehat{f_l\sigma}\|^{3/2}_3 \right)^{2/3}\\
&\lesssim \left( \sum_j\sum_{(k,l)} 2^{-j/2} \|f_k\|^{3/2}_{3/2} \|f_l\|^{3/2}_{3/2}\right)^{2/3}\\
&\lesssim \left( \sum_j\sum_k 2^{-j/2} \|f_k\|^3_{3/2}\right)^{2/3};
\end{split}
\end{equation}to pass to the last inequality, we have used that for given $k$, there are at most $O(1)$'s $\mathcal{C}^j_l\sim \mathcal{C}^j_k$.

By interpolation, for any $3/2<p<2$,
$$ 2^{-j/2} \|f_k\|^3_{3/2} \le 2^{-j/2} \|f_k\|_1^{3(1-\theta)}\|f_k\|^{3\theta}_p \le \left(2^{-j/2}\|f_k\|_1\right)^{3(1-\theta)}2^{-j(-1+3\theta/2)}\|f_k\|^{3\theta}_p,$$
where $\theta=p/3(p-1)$. So if normalizing $\|f\|_2=1$ and taking $\alpha=1-\theta$, it suffices to show that
$$ \sum_j\sum_k 2^{-j(3\theta/2-1)}\|f_k\|^{3\theta}_p \lesssim 1.$$
We decompose $f_k=f_k\chi_{|f|\le 2^{-j/2}}+f_k\chi_{|f|> 2^{-j/2}}=:f_k^{-}+f_k^{+}$. For $f_k^{-}$,
$$ \sum_j\sum_k 2^{-j(3\theta/2-1)}\|f_k^{-}\|^{3\theta}_p \lesssim 1.$$
Indeed, we apply H\"older's inequality with $(3\theta/p, 3\theta/(3\theta-p))$,
$$\sum_k \|f_k^{-}\|^{3\theta}_p \le \sum_k\int_{\mathcal{C}_k} |f|^{3\theta} (2^j)^{\frac {3\theta-p}{p}}\le \int |f|^{3\theta} (2^j)^{\frac {3\theta-p}{p}}.$$
Since $3\theta/2-1 -\frac {3\theta-p}{p} <0$ as $p<2$,
\begin{equation}\begin{split}
\sum_j\sum_k 2^{-j(3\theta/2-1)}\|f_k^{-}\|^{3\theta}_p &\le \int |f|^{3\theta} \sum_{j\le 0: |f|<2^{-j/2}} 2^{-j(3\theta/2-1 -\frac {3\theta-p}{p})}\\
&=C\int |f|^{3\theta} \times \sum_{j\le 0: 2^j< |f|^{-2}} 2^{j3\theta(\frac1p-\frac12)} \\
&\le C\int |f|^{3\theta} \times |f|^{3\theta-\frac {6\theta}p}
\le C\int |f|^{6\theta-\frac {6\theta}p}\\
&= C\int |f|^2<\infty.
\end{split}
\end{equation}
For $f_k^+$, we estimate it as follows: as $p<2$ and $3\theta=p'$, then $3\theta/p=p'/p>1$; then
\begin{equation}
\begin{split}
&\sum_{j,k} 2^{-j(\frac {3\theta}{2}-1)} \|f_k^+\|_p^{3\theta}=\sum_{j,k} 2^{-j(\frac {p'}{2}-1)} \|f_k^+\|_p^{p'}\\
&\le \left(\sum_{j,k} 2^{-jp(\frac {1}{2}-\frac {1}{p'})} \|f_k^+\|_p^p\right)^{p'/p}=\left(\sum_{j,k} 2^{-j\frac {2-p}{2}} \|f_k^+\|_p^p\right)^{p'/p}\\
&=\left(\sum_{j} 2^{-j\frac {2-p}{2}}\int_{|f|>2^{-j/2}} |f|^p \right)^{p'/p}=\left(\int  |f|^p \sum_{j\le 0:\,|f|>2^{-j/2}} 2^{-\frac j2 (2-p)} \right)^{p'/p}\\
&\le C\left(\int |f|^2 \right)^{p'/p}\le C.
\end{split}
\end{equation}
Hence Proposition \ref{prop:refined-Tomas-Stein} follows.
\end{proof}

Then Proposition \ref{prop-refinement-appliedto-near-extremal} follows by a similar argument as in \cite{Christ-Shao:extremal-for-sphere-restriction-I-existence, Shao:2008:linear-profile-Airy-Maximizer-Airy-Strichartz}.

Now we turn to show that ``distant caps interact weakly". We start with defining the distance between caps. The distance $\rho$ between two given caps $\mathcal{C}(z,r)$ and $\mathcal{C'}(z',r')$ is
$$\frac {r}{r'}+\frac {r'}{r}+\frac {|z-z'|}{r}.$$ For any metric space $(X,\rho)$ and any equivalence relation $\equiv$ on $X$, recall from the basic algebra the function $\rho([x],[y])=\inf_{x'\in [x],y'\in [y]}\rho(x',y')$ is a metric on the set of equivalence classes $X/\equiv$. Let $\mathcal{M}$ be the set of all caps $\mathcal{C}\subset S^1$ modulo the equivalence relation $\mathcal{C}\equiv -\mathcal{C}$, where $-\mathcal{C}=\{-x,x\in \mathcal{C}\}$. Then a metric on $\mathcal{M}$ can be defined in the following way.
\begin{definition}For any two caps $\mathcal{C}$, $\mathcal{C'}\in S^1$,
\begin{equation}
\varrho([\mathcal{C}], [\mathcal{C'}])=\min\bigl(\rho(\mathcal{C},\mathcal{C'}),\rho(-\mathcal{C},\mathcal{C'})\bigr),
\end{equation}where $[\mathcal{C}]$ denotes the equivalence class $[\mathcal{C}]=\{\mathcal{C},-\mathcal{C}\}\in \mathcal{M}$.
\end{definition} We will write $\varrho(\mathcal{C}, \mathcal{C'})=\varrho([\mathcal{C}], [\mathcal{C'}])$.

\begin{lemma}\label{le-distant-interaction}
For any $\eps>0$ there exists $\rho<\infty$ such that
\begin{equation}\label{eq-18}
\|\chi_\mathcal{C}\sigma\ast\chi_\mathcal{C}\sigma\ast \chi_{\mathcal{C'}}\sigma\|_{L^2} \le \eps |\mathcal{C}| |\mathcal{C'}|^{1/2}
\end{equation} whenever
\begin{equation}
\varrho(\mathcal{C}, \mathcal{C'}) >\rho.
\end{equation}
\end{lemma}
\begin{proof}
Set
$$ f=|\mathcal{C}|^{-1/2} \chi_\mathcal{C}\le Cr^{-1/2} \chi_\mathcal{C},\,\tilde{f}=|\mathcal{\tilde{C}}|^{-1/2} \chi_\mathcal{\tilde{C}}\le C\tilde{r}^{-1/2} \chi_\mathcal{\tilde{C}}.$$
Assume $\tilde{r}\le r$ and $\tilde{r}\ll 1$ and consider the case where no points are nearly antipodals. The case where the points are antipodals can be reduced to the previous case by using the identity
\begin{equation}\label{eq-81}
\langle f_1\sigma \ast f_2\sigma\ast f_3\sigma,  f_4\sigma \ast f_5\sigma \ast f_6\sigma \rangle =\langle f_1\sigma \ast f_2\sigma\ast \tilde{f_4}\sigma,  \tilde{f_3}\sigma \ast f_5\sigma \ast f_6\sigma\rangle
\end{equation} for any non-negative functions $f_j\in L^2(S^1)$, $j=1,\cdots, 6$. Let $\eps $ be given. It suffices to show, if $\varrho(\mathcal{C}, \mathcal{C'})\ge \rho$,
\begin{equation}\label{eq-19}
\|\chi_\mathcal{C}\sigma \ast \chi_{\mathcal{C'}}\sigma\|_{L^{3/2}} \le \eps |\mathcal{C}|^{1/2} |\mathcal{C'}|^{1/2}.
\end{equation}

{\bf Case I.} Assume $r\sim \tilde{r}$ and $|x-\tilde{x}|\ge 10 r$. Observe that
\begin{equation}
\|f\sigma \ast \tilde{f}\sigma \|_\infty\le (r\tilde{r})^{-1/2}\frac {C}{|x-\tilde{x}|},
\end{equation} and $f\sigma \ast \tilde{f}\sigma$ is supported in a rectangle with width $r$ and height $r|x-\tilde{x}|.$ Then
\begin{equation}
\|f\sigma \ast \tilde{f}\sigma\|_{L^{3/2}(\R^2)} \le (r\tilde{r})^{-1/2}\frac {C}{|x-\tilde{x}|} (r^2|x-\tilde{x}|)^{2/3} =C\left(\frac {r}{|x-\tilde{x}|}\right)^{1/3}\ll 1.
\end{equation}

{\bf Case II.} Assume $\tilde{r}\ll r$ and $|x-\tilde{x}|\ge 10 r$.
Still there holds
\begin{equation}
\|f\sigma \ast \tilde{f}\sigma \|_\infty\le (r\tilde{r})^{-1/2}\frac {C}{|x-\tilde{x}|},
\end{equation} but $f\sigma \ast \tilde{f}\sigma$ is supported in a tube with base length $r$ and width $\tilde{r}|x-\tilde{x}|.$ Then
\begin{equation}
\|f\sigma \ast \tilde{f}\sigma\|_{L^{3/2}(\R^2)} \le (r\tilde{r})^{-1/2}\frac {C}{|x-\tilde{x}|} (r\tilde{r}|x-\tilde{x}|)^{2/3} \le C\frac {(r\tilde{r})^{1/6}}{|x-\tilde{x}|^{1/3}}\le C(\tilde{r}/r)^{1/6}\ll 1.
\end{equation}

{\bf Case III.} Assume that $\tilde{r}\ll r$ and $|x-\tilde{x}|\le 10 r$. The analysis in Case II almost applies if $|x-\tilde{x}|\ge cr$ for some universal constant $c>0$, while it breaks down when their centers are very close. To cope with this difficulty, we employ the trick in \cite{Christ-Shao:extremal-for-sphere-restriction-I-existence}: we replace $f$ by its restriction $F$ to the complement of the cap $\mathcal{\tilde{C}}^\star$ centered at $\tilde{x}$ of radius $10r^{3/4}\tilde{r}^{1/4}$. Then
\begin{equation}
\|f-F\|_2 \le Cr^{-1/2}(r^{3/4}\tilde{r}^{1/4})^{1/2}\le C(\tilde{r}/r)^{1/8}\ll 1.
\end{equation}
Then we observe that
\begin{equation}
\|F\sigma \ast \tilde{f}\sigma \|_\infty\le (r\tilde{r})^{-1/2}\frac {C}{|x-\tilde{x}|},
\end{equation} but $F\sigma \ast \tilde{f}\sigma$ is supported in a tube with base length $r$ and width $\tilde{r}|x-\tilde{x}|$, whose area is less than $Cr\tilde{r}|x-\tilde{x}|$. Then
\begin{equation}
\begin{split}
\|F\sigma \ast \tilde{f}\sigma\|_{L^{3/2}(\R^2)} &\le (r\tilde{r})^{-1/2}\frac {C}{|x-\tilde{x}|} (r\tilde{r}|x-\tilde{x}|)^{2/3} \\
&\le C\frac {(r\tilde{r})^{1/6}}{|x-\tilde{x}|^{1/3}}\le C\frac {(r\tilde{r})^{1/6}}{\bigl(r^{3/4}\tilde{r}^{1/4}\bigr)^{1/3}} \le C\bigl(\frac {\tilde{r}}{r}\bigr)^{1/12}\ll 1.
\end{split}
\end{equation}
\end{proof}

\section{Step 4. On near-extremals: Proposition \ref{prop-prelimi-decomp-near-extremals}.}\label{sec:prelimi-decomp}
In this section, we aim to prove Proposition \ref{prop-prelimi-decomp-near-extremals}, which roughly speaking states that any nearly extremal to the inequality satisfies some appropriately scaled upper bounds relative to some cap up to a small error in $L^2$. As remarked in \cite{Christ-Shao:extremal-for-sphere-restriction-I-existence}, its proof is largely a formal argument relying on two inputs, Lemma \ref{prop-refinement-appliedto-near-extremal} and Lemma \ref{le-distant-interaction} which are already established in the previous step. We begin with
\begin{lemma}\label{le-1st-rule}
Let $f=g+h\in L^2(S^1)$. Suppose that $\langle g,h \rangle =0$ and $g\neq 0$, and that $f$ is a $\delta$-nearly extremal for some $\delta\in (0,\frac 14]$. Then
\begin{equation}\label{eq-29}
\frac {\|h\|_2}{\|f\|_2}\le C\max\left(\frac {\|h\sigma\ast h\sigma\ast h\sigma\|_2^{1/3}}{\|h\|_2},\delta^{1/2}\right),
\end{equation}where $0<C<\infty$ is a constant independent of $g$ and $h$.
\end{lemma}
The proof of this lemma is similar to \cite[Lemma 7.1]{Christ-Shao:extremal-for-sphere-restriction-I-existence} and will be omitted.

\subsection{A decomposition algorithm}
Let $f\in L^2(S^1)$ be a nonnegative function with positive norm. The same algorithm in \cite{Christ-Shao:extremal-for-sphere-restriction-I-existence} applies:
$$f=\sum_{0\le k\le \nu} f_k+G_{\nu+1}, \,\nu=0,1,\ldots$$
with the following properties.
\begin{itemize}
\item $G_0:=f$ and $\eps_0=1/2$. The inputs for Step $\nu$ are a nonnegative function $G_\nu\in L^2(S^1)$ and a positive number $\eps_\nu$. The outputs are functions $f_\nu$ and $G_{\nu+1}$ and nonnegative numbers $\eps_\nu^*$ and $\eps_{\nu+1}$.
    \begin{itemize}
    \item If $\|G_\nu\sigma \ast G_\nu\sigma \ast G_\nu\sigma \|_2=0$, then $G_\nu=0$ almost everywhere. The algorithm then terminates and we define $\eps_\nu^*=0$, $f_\nu=0$, and $G_\nu=0, f_\mu=0 $ and $\eps_\mu=0$ for $\mu> \nu$.
    \item If $0<\|G_\nu\sigma \ast G_\nu\sigma \ast G_\nu\sigma \|_2<\eps_\nu^3 \mathbf{S}^3 \|f\|_2^3$, we replace $\eps_\nu$ by $\eps_\nu/2$, and repeat until the first time that $\|G_\nu\sigma \ast G_\nu\sigma \ast G_\nu\sigma \|_2\ge \eps_\nu^3 \mathbf{S}^3 \|f\|_2^3$. Define $\eps_\nu^*$ to be this value of $\eps_\nu$. Then
  \begin{equation}\label{eq-51}
  (\eps^*_\nu)^3\mathbf{S}^3\|f\|_2^3 \le \|G_\nu\sigma \ast G_\nu\sigma \ast G_\nu\sigma \|_2\le 8(\eps^*_\nu)^3\mathbf{S}^3\|f\|_2^3.
  \end{equation}
Then an application of Lemma \ref{prop-refinement-appliedto-near-extremal} yields a decomposition for $G_\nu$: namely we obtain a cap $\mathcal{C}_\nu$ and $G_\nu=f_\nu+G_{\nu+1}$, where $f_\nu$ and $G_{\nu+1}$ satisfy the properties listed in that Lemma. We remark that the constants $C_\nu\le C(\eps_\nu^*)^{-O(1)}$ and $\eta_\nu\ge C(\eps_\nu^*)^{O(1)}$.  Define $\eps_{\nu+1}=\eps_\nu^*$ and move on to the next step $\nu+1$.
 \end{itemize}

\item If $f$ is even, then $f_\nu$ may likewise be chosen to be even.

\item If the algorithm terminates at some finite step $\nu$, then we have a finite decomposition $f=\sum_{0\le k\le \nu}f_k$ and $\eps_k^*=0$ for $k\ge \nu+1$.

\item  If the algorithm never terminates, then $\nu_N^*\to 0$ and $\sum_{0\le \nu\le N}f_\nu\to f$ in $L^2(S^1)$ as $N\to \infty$.
\end{itemize}

The algorithm yields some useful information when the decomposition algorithm is applied to nearly extremals. This is a consequence of Lemma \ref{prop-refinement-appliedto-near-extremal} and Lemma \ref{le-1st-rule}.
\begin{lemma}\label{le-2}
There exists a continuous function $\theta:(0,1]\to (0,\infty)$ such that for any $\eps>0$, there exists a $\delta>0$ such that for any nonnegative $\delta$-nearly extremal $f$ with $\|f\|_{L^2(S^1)}=1$. Then
$$\|f_\nu\|_2 \ge \theta(\|G_\nu\|_2)$$
for any index $\nu$ such that $\|G_\nu\|_2\ge \eps$.
\end{lemma}
\begin{proof}Let $C$ be the exact constant appearing in Lemma \ref{le-1st-rule}. Given $\eps>0$. We choose $\delta>0$ such that $C\delta^{1/2}<\eps/2$. Then the second alternative in Lemma \ref{le-1st-rule} yields that
$$ \frac {\|G_\nu\|_2}{\|f\|_2} \le \frac {C\|G_\nu\sigma \ast G_\nu\sigma \ast G_\nu\sigma \|_2^{1/3}}{\|G_\nu\|_2},$$
that is to say,
$$\|G_\nu\sigma \ast G_\nu\sigma \ast G_\nu\sigma \|_2 \ge \left(C^{-1} \|G_\nu\|_2^3\right)\|G_\nu\|_2^3.$$
Then an application of Lemma \ref{prop-refinement-appliedto-near-extremal} yields that, there exists a function $\theta:(0,1]\to (0,\infty)$ such that
$$\|f_\nu\|_2\ge \eta_{\|G_\nu\|_2} \|G_\nu\|_2 =:\theta (\|G_\nu\|_2)$$
where $\eta$ is as in Lemma \ref{prop-refinement-appliedto-near-extremal} and $\theta(x)$ can be regarded as $O(x^{O(1)})$. \end{proof}
Moreover, if $f$ is nearly extremal, the norms of $f_\nu$ and $G_\nu$ enjoy upper bounds independent of $f$ for all except for large $\nu$.
\begin{lemma}\label{le-3}
There exists a sequence of positive constants $\gamma_\nu\to 0$ and a function $N: (0,
\frac 12 ]\to \mathbb{Z}^+$ satisfying $N(\delta)\to \infty$ as $\delta\to 0$ such that for any nonnegative $f\in L^2(S^1)$, if $f$ is $\delta$-nearly extremal then the quantities $\eps^*_\nu$ obtained when the decomposition algorithm is applied to $f$ satisfy
\begin{align}
&\label{eq-52} \|G_\nu\|_2\le \gamma_\nu \|f\|_2 \text{ for all } \nu\le N(\delta),\\
&\label{eq-53} \eps_\nu^*\le \gamma_\nu \text{ for all } \nu\le N(\delta),\\
&\label{eq-54} \|f_\nu\|_2 \le \gamma_\nu \|f\|_2 \text{ for all } \nu\le N(\delta).
\end{align}
\end{lemma}
\begin{proof}The proof will be similar to that in \cite[Lemma 8.3]{Christ-Shao:extremal-for-sphere-restriction-I-existence}. We normalize $\|f\|_2=1$.
Since $\|f_\nu\|_2\le \|G_\nu\|_2$, and \eqref{eq-51} yields that
$$(\eps_\nu^*)^3 \mathbf{S}^3 \le \|G_\nu\sigma \ast G_\nu\sigma \ast G_\nu\sigma \|_2 \le \mathbf{S}^3 \|G_\nu\|_2^3,$$
we see that \eqref{eq-52} will imply both \eqref{eq-53} and \eqref{eq-54}. Hence we focus on proving \eqref{eq-52}. Let $\eta$ be the function appearing in Lemma \ref{prop-refinement-appliedto-near-extremal} and we know that $\eta(\delta)=O(\delta^{O(1)})$.

We first choose $\gamma_\nu$ such that it tends to zero so slowly that \begin{equation}\label{eq-55}
\nu \gamma_\nu^2\eta^2(c_0\gamma_\nu^3)>2 \text{ for all }\nu,
\end{equation} where $c_0$ will be clear below; it is possible since $ \eta(\delta)$ is in form of $O(\delta^{O(1)})$. Then given $\delta>0$, we choose $N(\delta)$ to be the least $\nu$ such that
$$ \gamma_\nu \le C\delta^{1/2},$$
where $C$ is the exact constant appearing in Lemma \ref{le-1st-rule}; we see that $N(\delta)\to \infty$ as $\delta\to 0$ and that $\gamma_\nu >C\delta^{1/2}$ for all $\nu\le N(\delta)$. Obviously the choices of $\gamma_\nu$ and $N$ are independent of $f$.

Now let $f$ and $\delta>0$ be given. If there were $\nu$ such that $\nu\le N(\delta)$ and $\|G_\nu\|_2 \ge \gamma_\nu$, then
$$ \|G_\nu \|_2\ge \gamma_\nu >C\delta^{1/2}.$$
Then Lemma \ref{le-1st-rule} yields that
 $$ \|G_\nu\|_2 \le \frac {C\|G_\nu\sigma \ast G_\nu\sigma \ast G_\nu\sigma \|_2}{ \|G_\nu\|_2}.$$
In other words, $$ \|G_\nu\sigma \ast G_\nu\sigma \ast G_\nu\sigma \|_2 \ge \left(c_0\|G_\nu\|_2^3\right)\|G_\nu\|_2^3\ge \left(c_0\gamma_\nu^3\right)\|G_\nu\|_2^3.$$
 where $c_0=C^{-1}$. Then Lemma \ref{prop-refinement-appliedto-near-extremal} yields that
 $$\|f_\nu\|^2_2\ge \eta^2(c_0\gamma_\nu^3)\|G_\nu\|_2^2 \ge \eta^2(c_0\gamma_\nu^3) \gamma_\nu^2.$$
Since $\|G_\mu\|_2\ge \|G_\nu\|_2$ for $\mu\le \nu$, there also holds that $\|G_\mu\|_2> C\delta^{1/2}$ for $\mu\le \nu$. So one may repeat the procedure above and find that
$$ \|f_\mu\|^2_2\ge \eta^2(c_0\gamma_\nu^3)\gamma_\nu^2 \text{ for } \mu\le \nu.$$
On the other hand, we have $\sum_{0\le \mu\le \nu} \|f_\mu\|_2^2 \le \|f\|_2^2=1$, which gives
$$ \sum_{\mu\le \nu } \eta^2(c_0\gamma_\nu^3)\gamma_\nu^2 \le 1, \Rightarrow \nu \gamma_\nu^2\eta^2(c_0\gamma_\nu^3) \le 1. $$
This is a contradiction to the choice of $\gamma_\nu$ in \eqref{eq-55}. So we finish the proof of this lemma.
\end{proof}

\subsection{A geometric property of the decomposition}
In the previous subsection, based on the single analytic fact Lemma \ref{prop-refinement-appliedto-near-extremal}, we have established that the $L^2$-norms of $f_\nu$ and $G_\nu$ obtained when the decomposition algorithm is applied to nearly extremals $f$ satisfy some uniform upper bounds as in Lemma \ref{le-2} and Lemma \ref{le-3}. On the other hand, in Lemma \ref{le-distant-interaction}, we have proved that ``distant caps interact weakly"; this will provide us some additional information on near extremals. We first recall a lemma in \cite[Lemma 9.1]{Christ-Shao:extremal-for-sphere-restriction-I-existence}, which follows easily from the pigeonhole principle.
\begin{lemma}\label{le-5}
In any metric space, for any $N$ and $r$, any finite set $S$ of cardinality $N$ and diameter equal to $r$ may be partitioned into two disjoint non-empty subsets $S=S_1\cup S_2$ such that the distance of $S_1$ and $S_2$ is no less than $r/2N$. Moreover, given two points $s_1,s_2\in S$ satisfying $\operatorname{distance}(s_1,s_2)=r$, this partition may be constructed such that $s_1\in S_1$ and $s_2\in S_2$.
\end{lemma}
As a consequence, we have
\begin{lemma}\label{le-6}
For any $\eps>0$, there exists $\delta>0$ and $0< \lambda<\infty$ such that for any nonnegative $\delta$-nearly extremal $f$, the summands $f_\nu$ obtained from the decomposition algorithm and the associated caps $\mathcal{C}_\nu$ satisfy
\begin{equation}\label{eq-57}
\varrho(\mathcal{C}_j,\mathcal{C}_k) \le \lambda, \text{ whenever } \|f_j\|_2\ge \eps\|f\|_2 \text{ and }\|f_k\|_2\ge \eps\|f\|_2.
\end{equation}
\end{lemma}
\begin{proof} We follow the proof from \cite{Christ-Shao:extremal-for-sphere-restriction-I-existence}. We normalize $\|f\|_2=1$ and let $\eps$ be given; also suppose that $\|f_{j_0}\|_2, \|f_{k_0}\|_2\ge \eps$. Let $N$ be the smallest integer such that $\|G_{N+1}\|_2 <\eps^3$. This choice of $N$ may depend on $f$ but it will not affect our final choice of $\lambda$. Since $\|f_\nu\|_2 \le \|G_\nu\|_2$ and $\|G_\nu\|_2$ is a non-increasing function of $\nu$, we see that $j_0,k_0\le N$. Moreover Lemma \ref{le-3} yields that there exists a $M_\eps$ which depends only on $\eps$ such that $N\le M_\eps$. If we choose $\delta$ to be sufficiently small but depending on $\eps$, we see that Proposition \ref{prop-refinement-appliedto-near-extremal} yields that $f_\nu \le \theta(\eps)|\mathcal{C}_\nu|^{-1/2} \chi_{\mathcal{C}_\nu\cup -\mathcal{C}_\nu}$, where $\theta$ is a continuous, strictly positive function on $(0,1]$. For those $\nu\le N$, we have $\|G_\nu\|_2 \ge \eps^3$.

Now let $0<\lambda<\infty$ to be a large quantity to be specified. It suffices to show that if $\delta(\eps)$ is sufficiently small, an assumption that $\varrho(\mathcal{C}_j,\mathcal{C}_k)>\lambda$ would lead to a contradiction that $f$ is a $\delta$-nearly extremal.

We apply the previous Lemma \ref{le-5} to see that
$$F=F_1+F_2:=\sum_{\nu\in S_1} f_\nu+ \sum_{\nu\in S_2} f_\nu,$$
where $[0,N] =S_1\cup S_2$, $j_0\in S_1$ and $k_0\in S_2$; also we have $\varrho(\mathcal{C}_j,\mathcal{C}_k)\ge \frac {\lambda}{2N} \ge \frac {\lambda}{2M_\eps}$ for all $j\in S_1, k\in S_2$.  For $i,j\in \{1,2\} $ and $i\neq j$,
$$ \|F_i\sigma \ast F_i\sigma \ast F_j\sigma \|_2 \le \sum_{j\in S_1, k\in S_2} \|f_j\sigma \ast f_j\sigma \ast f_k\sigma \|_2 \le M_\eps^3 \gamma(\frac {\lambda}{2M_\eps})\theta^3(\eps),$$
where $\gamma(t)\to0$ as $t\to \infty$ as in Lemma \ref{le-distant-interaction}.
Therefore,
\begin{equation}\label{eq-58}
\begin{split}
&\|F\sigma \ast F\sigma \ast F\sigma \|_2^2 \\
&\le \|F_1\sigma \ast F_1\sigma \ast F_1\sigma \|_2^2 +\|F_2\sigma \ast F_2\sigma \ast F_2\sigma \|_2^2+\sum_{i\neq j, \atop i,j\in \{1,2\}} \|F_i\sigma \ast F_i\sigma \ast F_j\sigma \|_2^2 \|f\|^3_2\\
&\le \mathbf{S}^6 \left(\|F_1\|_2^6+\|F_2\|_2^6\right)+CM_\eps^3\gamma(\frac {\lambda}{2M_\eps})\theta^3(\eps)\\
&\le \mathbf{S}^6 \left(\|F_1\|_2^2+\|F_2\|_2^2\right)\max\{\|F_1\|_2^4,\|F_2\|_2^4\}+CM_\eps^3\gamma(\frac {\lambda}{2M_\eps})\theta^3(\eps)\\
&\le \mathbf{S}^6 (1-\eps^4)+CM_\eps^3\gamma(\frac {\lambda}{2M_\eps})\theta^3(\eps),
\end{split}
\end{equation}where we have used $\|F_1\|_2\ge \eps$ and $\|F_2\|_2\ge \eps$ in passing to the last inequality.
On the other hand,
\begin{equation}\label{eq-59}
\begin{split}
(1-\delta)^6\mathbf{S}^6 &\le \|f\sigma \ast f\sigma \ast f\sigma \|_2^2\le \|F\sigma \ast F\sigma \ast F\sigma \|_2^2+C\|f\|_2^4\|f-F\|_2^2\\
&\le \|F\sigma \ast F\sigma \ast F\sigma \|_2^2+C\eps^6.
\end{split}
\end{equation}
So from \eqref{eq-58} and \eqref{eq-59}, we see that
\begin{equation}\label{eq-60}
(1-\delta)^6\mathbf{S}^6 \le C\eps^6 +\mathbf{S}^6(1-\eps^4)+CM_\eps^3\gamma(\frac {\lambda}{2M_\eps})\theta^3(\eps).
\end{equation}
Recall that $\gamma(t) \to 0$ as $t\to\infty$. So given $\eps>0$ which is small, if we chose a sufficiently small $\delta=\delta(\eps)$, then \eqref{eq-60} will result in a contradiction if $\lambda$ is allowed to be sufficiently large. Hence the conclusion of the lemma follows.
\end{proof}

\subsection{Upper bounds for extremizing sequences}
In this subsection, we prove Proposition \ref{prop-prelimi-decomp-near-extremals}.
\begin{lemma}\label{le-7}
There exists a function $\Theta: [1,\infty)\to (0,\infty)$ satisfying $\lim_{R\to \infty} \Theta(R)=0$ such that the following holds: given any $\eps>0$ and $\bar{R}>0$, there exists $\delta>0$ such that for any nonnegative $\delta$-nearly extremal $f$ with $\|f\|_2=1$, we have a decomposition
$$f=F+G$$
where $F$ and $G$ are even and nonnegative with disjoint supports. Moreover this decomposition satisfies
$ \|G\|_2\le \eps$ and there exists a cap $\mathcal{C}=\mathcal{C}(z,r)$ such that for any $R\in [1,\bar{R}]$, we have \begin{align}
&\label{eq-61} \int_{\min\{|x+z|,|x-z|\}\ge Rr} F^2(x)d\sigma(x) \le \Theta(R), \\
&\label{eq-62} \int_{\{F(x)\ge Rr^{-1/2} \}} F^2(x) d\sigma(x) \le \Theta(R).
\end{align}
\end{lemma}
Let us postpone the proof of this lemma; now we prove Proposition \ref{prop-prelimi-decomp-near-extremals} by using it. \begin{proof}[Proof of Proposition \ref{prop-prelimi-decomp-near-extremals} from Lemma \ref{le-7}] Let $\eps$ and $f$ be given. We assume that the $\Theta$ given by Lemma \ref{le-7} is a continuous, strictly decreasing function. Define $\bar{R}=\bar{R}(\eps)$ by the equation $\Theta(\bar{R})=(\eps/2)^2$. Let $\mathcal{C}=\mathcal{C}(z,r)$ and $\delta=\delta(\eps, \bar{R}(\eps))$ along with $F$ and $G$ satisfy the conclusions in Lemma \ref{le-7}. We re-define
\begin{equation}\label{eq-63}
f=((1-\chi) F) +(\chi F +G)=:\tilde{F}+\tilde{G} ,
\end{equation} where $\chi(x)=1$ if $\min\{|x-z|,|x+z|\} \ge \bar{R}r$ or $F(x)\ge \bar{R} r^{-1/2}$. Then it easily follows that
\begin{equation}\label{eq-64}
\|\tilde{G}\|_2 \le \|G\|_2+ \|\chi F\|_2 \le \eps+2\times \eps/2 =2\eps.
\end{equation}
On the other hand, we also have
\begin{equation}\label{eq-65}
\begin{split}
\int_{\min\{|x+z|,|x-z|\}\ge Rr} \tilde{F}^2(x)d\sigma(x) &\le \Theta(R), \\
\int_{\{\tilde{F}(x)\ge Rr^{-1/2} \}} \tilde{F}^2(x) d\sigma(x) &\le \Theta(R).
\end{split}
\end{equation}
Indeed, when $R\le \bar{R}$, $\tilde{F}\le F$,
\begin{equation}\label{eq-66}
\begin{split}
\int_{\min\{|x+z|,|x-z|\}\ge Rr} \tilde{F}^2(x)d\sigma(x) &\le\int_{\min\{|x+z|,|x-z|\}\ge Rr} F^2(x)d\sigma(x) \le \Theta(R), \\
\int_{\{\tilde{F}(x)\ge Rr^{-1/2} \}} \tilde{F}^2(x) d\sigma(x) &\le \int_{\{F(x)\ge Rr^{-1/2} \}} F^2(x) d\sigma(x)\le  \Theta(R);
\end{split}
\end{equation}
when $R\ge \bar{R}$, from the support information of $\chi$,
\begin{equation}\label{eq-67}
\begin{split}
\int_{\min\{|x+z|,|x-z|\}\ge Rr} \tilde{F}^2(x)d\sigma(x) &\le \int_{\min\{|x+z|,|x-z|\}\ge \bar{R}r} (1-\chi)F^2(x)d\sigma(x)=0, \\
\int_{\{\tilde{F}(x)\ge Rr^{-1/2} \}} \tilde{F}^2(x) d\sigma(x) &\le \int_{\{F(x)\ge \bar{R}r^{-1/2} \}} (1-\chi)F^2(x) d\sigma(x) =0.
\end{split}
\end{equation}Hence the proof of Proposition \ref{prop-prelimi-decomp-near-extremals} is complete if we assume Lemma \ref{le-7}.
\end{proof}
We are left with proving Lemma \ref{le-7}.
\begin{proof}[Proof of Lemma \ref{le-7}] Let $\eps>0$ and $f\ge 0$ be given with $\|f\|_2=1$ and also $R\in [1,\bar{R}]$. Let $\{f_\nu, G_\nu\}$ be the pairs obtained from the decomposition algorithm. Choose $\delta=\delta(\eps)$ sufficiently small and $M=M(\eps)$ sufficiently large such that
\begin{equation}
\begin{split}
& \|G_{M+1}\|_2 \le \eps/2,\\
& f_\nu, G_\nu \text{ satisfy the conclusions in Lemma \ref{le-3} for all }\nu\le M.
\end{split}
\end{equation}
Set $F=\sum_{0\le \nu\le M} f_\nu$. Then $\|f-F\|_2 =\|G_{M+1}\|_2\le \eps/2$.  Let $\eta: [1,\infty)\to (0,\infty)$ be a function to be chosen in the end of the proof satisfying that $\eta(t)\to 0$ as $t\to\infty$. This function will not depend on $\bar{R}$.

Let $A(\eta):=\inf\{\nu: \|f_\nu\|_2 <\eta\}$. Then set $N:=\min\{M, A(\eta)\}$. Clearly from the upper bound on $A(\eta)$, $N$ is majorized by a quantity depending only on $\eta$ by Lemma \ref{le-3}. Set $\mathcal{F}=\mathcal{F}_N :=\sum_{0\le \nu\le N} f_\nu$. Then it follows from Lemma \ref{le-3} that
\begin{equation}\label{eq-68}
\|F-\mathcal{F}\|_2 \le \gamma(\eta),
\end{equation}where $\gamma(\eta)\to 0$ as $\eta\to 0$. The function $\eta$ is independent of $\epsilon$ and $\bar{R}$.

Let $\mathcal{C}_0=\mathcal{C}_0(z_0,r_0)$ be the cap associated to $f_0$ in the decomposition of $f$, and $\mathcal{C}_0$ will be the desired cap in Lemma \ref{le-7}. Then we need to find a function $\Theta$ to guarantee that both \eqref{eq-61} and \eqref{eq-62} hold; in this process, we need to choose a suitable function $\eta$. Suppose that the functions $R\mapsto \eta(R)$ and $R\mapsto \Theta(R)$ are chosen such that
\begin{equation}\label{eq-69}
\begin{split}
& \eta(R) \to 0, \text{ as } R\to \infty,\\
& \gamma(\eta(R))\le \Theta(R), \text{ for all } R.
\end{split}
\end{equation}
Then by \eqref{eq-68}, $F-\mathcal{F}$ satisfies the desired estimates \eqref{eq-61} and \eqref{eq-62}. Then it suffices to show that \begin{align}
&\label{eq-70} \mathcal{F}(x)=0, \text{ whenever }  \min\{|x+z_0|, |x-z_0|\} \ge Rr_0, \\
&\label{eq-71} \|\mathcal{F}\|_\infty \le Rr_0^{-1/2}.
\end{align}
Before proving \eqref{eq-70}, we recall several facts. Firstly each summand $f_k\le C(\eta)|\mathcal{C}_k|^{-1/2} \chi_{\mathcal{C}_k\cup -\mathcal{C}_k}$ where $C(\eta)<\infty$ depends only on $\eta$, and $f_k$ is supported by $\mathcal{C}_k\cup -\mathcal{C}_k$. Moreover, for all $k\le N$, $\|f_k\|_2 \ge \eta$ by the definition of $N$. Then an application of Lemma \ref{le-6} implies that there exists a function $\eta\mapsto \lambda(\eta) $ such that, if $\delta$ is sufficiently small as a function of $\eta$, we have $\varrho(\mathcal{C}_k, \mathcal{C}_0) \le \lambda(\eta)$ for all $k\le N$. This is needed for $\eta=\eta(R)$ for all $R$ in the compact set $[1,\bar{R}]$ so that $\delta$ can be chosen as a function of $\bar{R}$ alone. Hence $\delta$ may be chosen as a function of $\bar{R}$ in addition to the previous dependence on $\eps$.

We are ready to prove \eqref{eq-70}. Given $x\in S^1$ with $\min\{|x-z_0|, |x+z_0|\} \ge Rr_0$, either $f_k(x) =0$ or $\mathcal{C}_k$ has radius  $\ge \frac 14 Rr_0$, or the center $z_k$ of $\mathcal{C}_k$ satisfies that $\min\{|z_k+z_0|,|z_k-z_0|\} \ge \frac 14 R r_0$. In the latter two cases, there always holds that $\varrho(\mathcal{C}_k,\mathcal{C}_0)\ge CR$. So
\begin{equation}\label{eq-72}
R\le C \lambda(\eta(R)).
\end{equation}
This is a contradiction to the choice of $\eta$ if $\eta(R)\to 0$ slowly enough as $R\to \infty$. Then we have $\mathcal{F}(x)\equiv 0$ when $\min\{|x+z|, |x-z|\}\ge Rr_0$. With the choice of $\eta$, $\Theta$ can be defined by
\begin{equation}\label{eq-73}
\Theta(R):=\gamma(\eta(R)).
\end{equation}Then \eqref{eq-61} holds for all $R\in [1,\bar{R}]$.

Next we prove \eqref{eq-71}. We claim that $\|\mathcal{F}\|_\infty \le Rr_0^{-1/2}$ if $R$ is taken sufficiently large as a function of $\eta$. Indeed, because the summands $f_k$ have pairwise disjoint supports, it suffices to control $\max_{k\le N} \|f_k\|_\infty$. For this, Lemma \ref{prop-refinement-appliedto-near-extremal} implies that
$$\|f_k\|_\infty \le C(\eta) r^{-1/2}_k, C(\eta)=O(\eta^{-O(1)}).$$
If $\eta(R)$ is chosen to go to zero sufficiently slowly to ensure that $C(\eta(R))\lambda(\eta(R))<R$ for all $k\le N$, then \eqref{eq-62} holds provided that $\Theta$ is defined as in \eqref{eq-73}. Indeed, given any $k\le N$, $\|f_k\|_\infty\le Rr_0^{-1/2}$ would follow if $C(\eta(R))r_k^{-1/2} \le Rr_0^{-1/2}$; then it reduces to show that
\begin{equation}\label{eq-74}
 C(\eta(R))\lambda(\eta(R)) \le R
 \end{equation}
However \eqref{eq-74} is guaranteed if we choose $\eta(R)\to 0$ sufficiently slow as $R\to \infty$.

Finally $\eta$ must be chosen to tend to zero slowly enough to satisfy the requirements in \eqref{eq-72} and \eqref{eq-74}. With this choice of $\eta$, the proof of Lemma \ref{le-7} is complete.
\end{proof}

\section{Step 5. Ruling out small caps and existence of extremals}\label{sec:step5}
This step aims to establish the Proposition \ref{prop-regularity-after-rescaling}. We split the proof into 3 subsections, \ref{sec-key-prop}, \ref{sec-rule-out-small} and \ref{sec-big-caps}. In Subsection \ref{sec-key-prop}, we prove two propositions, one on the decomposition of $\widehat{f_\nu\sigma}$ into ``profiles", and the other on orthogonality of such profiles. Then in Subsection \ref{sec-rule-out-small}, we rule out the ``small caps" case where $\lim_{\nu\to\infty}r_\nu=0$ with the additional information that $\mathcal{R}>(5/2)^{1/6}\mathcal{R}_{\textbf{P}}$. Then we are left with ``large caps" case, i.e., where $\inf_\nu r_\nu >0$. In Subsection \ref{sec-big-caps}, we show that an extremal is obtained for \eqref{eq-1}.

Let $\{f_\nu\}$ be an even nonnegative extremizing sequence, uniformly upper even normalized with respect to the caps $\{\mathcal{C}_\nu\cup (-\mathcal{C}_\nu)\}$. Without loss of generality, we may assume that $\mathcal{C}_\nu$ is supported on the upper hemisphere of $S^1$, $S^1_+:= \{y\in S^1:~y\cdot (0,1)>0\}$. The sequence $\{f_\nu\}$ satisfies that $\|f_\nu\|_{L^2(S^1)} =1$. Suppose that $\inf_\nu r_\nu=0$. Then up to a subsequence, we may assume that $\lim_{\nu\to\infty}r_\nu=0.$  

Decompose $$2^{1/2}f_\nu(x) = f_\nu^{+}(x)+ f_\nu^+(-x)+f_\nu^\frak{b}(x),$$
where $f_\nu^+$ is real, $f_\nu^+$ is supported on $\mathcal{C}(z_\nu, r_\nu^{1/2})$, and $\|f_\nu^\frak{b}\|_{L^2}\to 0$ as $\nu\to \infty$.
 
Set $$g_\nu:=\phi_\nu^*(f^+_\nu)=r_\nu^{1/2} f^+_\nu(\phi_{\mathcal{C}_\nu})/(1-r_\nu^2y^2)^{\frac 14},$$ where $\phi^*_\nu$ is the rescaling map associated to $\mathcal{C}_\nu\cup (-\mathcal{C}_\nu)$. It is not hard to see that $g_\nu$ is upper normalized with respect to the unit ball $\mathbb{B}\subset \R$, i.e.,
\begin{equation}\label{eq-46}
\begin{split}
& g_\nu\ge 0,\\
& \|g_\nu\|_2\to 1, \text{ as }\nu\to\infty,\\
& \int_{|x|\ge R} |g_\nu|^2 dx \le \Theta(R),\,\forall R\ge 1,\\
& \int_{g_\nu \ge R} |g_\nu|^2dx \le \Theta(R),\,\forall R\ge 1,\\
& \Theta(R) \to 0, \text{ as }R\to \infty.
\end{split}
\end{equation}

Since $f_\nu$ is even and $\|f_\nu\|_{L^2}=1$, we have $\|f_\nu^+\|_{L^2}\to 1$ as $\nu\to \infty$. Moreover the function $F_\nu :=f^+_{\nu} (x)+f_\nu^{+}(-x)$ satisfies 
$$ \frac {\|F_\nu\sigma *F_\nu\sigma *F_\nu\sigma\|_{L^2}^2}{\|F_\nu\|_{L^2}^6}= \frac {20}{8} \frac {\|f_\nu^+\sigma *f_\nu^+\sigma *f_\nu^+\sigma\|_{L^2}^2}{\|f_\nu^+\|_{L^2}^6}.$$
So we have 
\begin{equation}\label{eq-c1}
 \limsup_{\nu\to \infty} \|f_\nu\sigma *f_\nu\sigma *f_\nu\sigma \|_{L^2}^2 = \frac {5}{2} \limsup_{\nu\to \infty} \|f_\nu^+\sigma *f_\nu^+\sigma *f_\nu^+\sigma \|_{L^2}^2. 
 \end{equation}

We will show that the right hand side of \eqref{eq-c1} is less than $\frac {5}{2}\textbf{P}^6$ by developing a profile decomposition for $f_\nu^+$. For simplicity of notations, in the subsection \ref{sec-key-prop} below, we will write $f_\nu$ as $f_\nu^+$.  

\subsection{Key propositions}\label{sec-key-prop}
Recall that we write $f_\nu^+$ as $f_\nu$ and $f_\nu^+$ is supported on $\mathcal{C}(z_\nu,r_\nu^{1/2})$. By rotation invariance, we may assume that $z_\nu=(0,1)$ for all $\nu$.
The decomposition for $\widehat{f_\nu \sigma}$ is motivated by the rescaling relation,
\begin{equation}\label{eq-b5}
\begin{split}
&\left|\widehat{f_\nu \sigma}(x,t)\right|=\left|\int_{\mathcal{C}_\nu} e^{i(x,t)\cdot\xi}f_\nu(\xi)d\sigma(\xi)\right|\\
&=\left|\int_{|y|\le 1/2} e^{ixy+it\sqrt{1-y^2}}\frac {f_\nu(y,\sqrt{1-y^2})}{\sqrt{1-y^2}}dy\right|\\
&=\left|r_\nu^{1/2} \int e^{ir_\nu xy+i t\sqrt{1-r_\nu^2 y^2}}\frac {r_\nu^{1/2} f_\nu(r_\nu y,\sqrt{1-r_\nu^2 y^2})}{\sqrt{1-r_\nu^2 y^2}}dy\right|\\
&=\left|r_\nu^{1/2} \int e^{ir_\nu xy-\frac {itr_\nu^2 y^2}{2}} e^{ir_\nu^2 t\bigl(\frac {\sqrt{1-r_\nu^2 y^2}-1}{r_\nu^2}+
\frac {y^2}{2}\bigr)}\frac {1}{(1-r_\nu^2 y^2)^{\frac 14}}\frac { r_\nu^{1/2}f_\nu(r_\nu y,\sqrt{1-r_\nu^2 y^2}) }{(1-r_\nu^2y^2)^{\frac 14}}dy\right|\\
&=\left|r_\nu^{1/2} e^{\frac {ir_\nu^2 t \Delta}{2}}\bigl(h_\nu(r_\nu^2 t,y) g_\nu(y)\bigr)(r_\nu x)\right|,
\end{split}
\end{equation}where
\begin{equation}\label{eq-32}
h_\nu(t,y):=e^{it\left( \frac {\sqrt{1-r_\nu^2|y|^2}-1}{r_\nu^2}+\frac {|y|^2}{2}\right)}.
\end{equation}
So a decomposition for $\widehat{f_\nu \sigma}$ immediately follows once we have a decomposition for $\{g_\nu\}$. Set $h_\nu^{-1}=1/h_\nu$.

\begin{proposition}\label{prop-decomp}
Let $\{g_\nu\}$ and $\{h_\nu\}$ be defined as above. Then there exists a sequence $(x_\nu^k,t_\nu^k)\in \R^2$ and $e_\nu^l\in L^2(\R)$ such that
\begin{equation}\label{eq-b3}
g_\nu(y)=\sum_{j=1}^l e^{\frac {it^j_\nu y^2}{2}}e^{-ix^j_\nu y} h^{-1}_\nu (t^j_\nu, y)\phi^j+e^l_\nu(y)
\end{equation} with the following properties: The parameters $\{(x_\nu^k,t_\nu^k)\}$ satisfy, for $k\neq j$,
\begin{equation}\label{eq-b4}
|x_\nu^k-x_\nu^j|+|t_\nu^k-t_\nu^j|\to \infty, \text{ as }\nu\to\infty.
\end{equation}
For each $l\ge 1$,
\begin{equation}\label{eq-L2-ortho}
\|f_\nu\|_{L^2(S^1)}^2=\sum_{j=1}^l\|\phi^j\|_2^2+\|e_\nu^l\|_2^2 \text{ as } \nu\to \infty.
\end{equation}
The function $e_\nu^l$ satisfies
\begin{equation}\label{eq-b8}
\limsup_{l\to\infty}\limsup_{\nu\to\infty}\left\| r_\nu^{1/2} e^{\frac {itr_\nu^2\Delta}{2}}\bigl[ h_\nu(r_\nu^2 t,y)\frac {e_\nu^l}{(1-r_\nu^2 y^2)^{\frac 14}}\bigr](r_\nu x)\right\|_{L^6_{t,x}(\R^2)}=0.
\end{equation}where
$$e^{\frac {it\Delta}{2}}f(x)=\int e^{ixy-\frac {ity^2}{2}}f(y)dy.$$
\end{proposition}

\begin{proposition}[Orthogonality]\label{prop-ortho}
Let $\{(x_\nu^j,t_\nu^j)\}$ be as above and set
\begin{align}
G_\nu^k& :=e^{\frac {it^k_\nu y^2}{2}}e^{-ix^k_\nu y} h^{-1}_\nu (t^k_\nu, y)\phi^k,\\
G_\nu^j& :=e^{\frac {it^j_\nu y^2}{2}}e^{-ix^j_\nu y} h^{-1}_\nu (t^j_\nu, y)\phi^j.
\end{align} Then for $k\neq j$,
\begin{equation}\label{eq-b6}
\begin{split}
&\lim_{\nu\to\infty} \left\| \Bigl(r_\nu^{1/2} e^{\frac {itr_\nu^2\Delta}{2}}\bigl[ h_\nu(r_\nu^2 t,y)\frac {G_\nu^k}{(1-r_\nu^2 y^2)^{\frac 14}} \bigr](r_\nu x)\Bigr)\Bigl(r_\nu^{1/2} e^{\frac {itr_\nu^2\Delta}{2}}\bigl[h_\nu(r_\nu^2 t,y)\frac {G_\nu^j}{(1-r_\nu^2 y^2)^{\frac 14}}  \bigr](r_\nu x)\Bigr) \right\|_{L^3_{t,x}(\R^2)}=0.
\end{split}
\end{equation}
\end{proposition}
We state a useful lemma on a localized linear restriction estimate., which will be used in the proof of Proposition \ref{prop-decomp}.
\begin{lemma}\label{le-local-restr}
Let $4<q<6$ and $h_\nu$ be defined as in \eqref{eq-32}. Assume that $\lim_{\nu\to\infty} r_\nu=0$. Then if $|(1-r_\nu^2y^2)^{1/4}f|\le M$ for some $M>0$ and for all $|y|\le R$ and all sufficiently large $\nu$,
$$ \|e^{\frac {it\Delta}{2}} [h_\nu(t,y) \frac {f(y)}{(1-r_\nu^2 y^2)^{1/4}}]\|_{L^q_{t,x}} \le CM, \text{ uniformly in sufficiently large } \nu,$$
where the constant may depend on $R$, but not on $\nu$.
\end{lemma}
\begin{proof} Choose $r_\nu$ sufficiently small such that $B(0,R) \subset \{|y| \le \frac 12 r_\nu^{-1}\}$.  We write
\begin{equation*}
\begin{split}
e^{\frac {it\Delta}{2}} [h_\nu(t,y) \frac{f(y)}{(1-r_\nu^2y^2)^{1/4}}](x)&=\int e^{ixy+it \frac {\sqrt{1-|r_\nu y|^2}-1}{r_\nu^2}}\frac {1}{(1-|r_\nu y|^2)^{\frac 14}}f(y)dy \\
& =\int e^{i\frac {x}{r_\nu} y+i\frac {t}{r_\nu^2} \sqrt{1-|y|^2}} r_\nu^{-1}f(r_\nu^{-1}y)(1-|y|^2)^{\frac 14}\frac {dy}{\sqrt{1-|y|^2}}.
\end{split}
\end{equation*}
Then
\begin{equation*}
\begin{split}
&\left\|\int e^{i\frac {x}{r_\nu} y+i\frac {t}{r_\nu^2} \sqrt{1-|y|^2}} r_\nu^{-1}f(r_\nu^{-1}y)(1-|y|^2)^{\frac 14} \frac {dy}{\sqrt{1-|y|^2}}\right\|_{L^q_{t,x}} \\
&= r_\nu^{-1+\frac 3q}\left\|\int_{S^1_+} e^{ixy+it\sqrt{1-|y|^2}} f(r_\nu^{-1}y)(1-|y|^2)^{\frac 14} d\sigma\right\|_{L^q_{t,x}}\\
& \le r_\nu^{-1+\frac 3q} \|f(r_\nu^{-1}y)(1-y^2)^{1/4}\|_{L^p(\sigma, +)},
\end{split}
\end{equation*} where $p$ satisfies $\frac 3q=1-\frac {1}{p}, p<2$, and $L^p(\sigma, +)$ is understood as integrating over $S^1_+$; we have also regarded $f(y)$ as a function on the upper hemisphere $S^1_+:=\{z\in S^1:\, z\cdot (0,1) >0\}$.

Then continuing the above inequality, we have
\begin{equation*}
\begin{split}
 & r_\nu^{-1+\frac 3q}\left( \int |f(r_\nu^{-1}y) (1-y^2)^{\frac 14}|^p\frac {dy}{\sqrt{1-y^2}} \right)^{1/p}  \\
 & \le r_\nu^{-1+\frac 3q+\frac 1p} \left( \int_{|y|\le R}  |f(y) (1-r_\nu^2y^2)^{\frac 14}|^p \frac {dy}{\sqrt{1-r_\nu^2y^2}}\right)^{1/p} \\
& \le CM R^{1/p},
\end{split}
\end{equation*} for all sufficiently large $\nu$. This finishes the proof of Lemma \ref{le-local-restr}.
\end{proof}

Now we will first prove Proposition of \ref{prop-decomp}, and then Proposition \ref{prop-ortho}.
\begin{proof}[The proof of Proposition \ref{prop-decomp}.] We split the proof into two steps.

\textbf{Step 1.} For $(x_\nu,t_\nu)\in \R^2$, we define
$$T_\nu(g)(y)=e^{-\frac {it_\nu y^2}{2}} e^{ix_\nu y}h_\nu(t_\nu,y) g(y);$$
analogously $T_\nu^i$ for $(x_\nu^i,t_\nu^i)$ for $i\ge 1$, and $T_\nu^{-1}(g)(y)=e^{\frac {it_\nu y^2}{2}} e^{-ix_\nu y}h^{-1}_\nu(t_\nu,y)g(y)$. Let $P^0$ denote the sequence $\{g_\nu\}_{\nu\ge 1}$. Then we define the set
$$\mathcal{W}(P^0)=\{w-\lim_{\nu\to\infty} T_\nu (P_\nu^0)(y) \text{ in } L^2(\R):\, (x_\nu,t_\nu)\in \R^2\},$$
where $w-\lim f_\nu$ denotes a weak limit of $\{f_\nu\}$ in $L^2$. Define the blow-up criterion associated to $\mathcal{W}(P^0)$:
$$ \mu(P^0):=\sup\{\|\phi\|_{L^2(\R)}: \phi\in \mathcal{W}(P^0)\}.$$

Then for any $\phi\in \mathcal{W}(P^0)$,
$$\|\phi\|_{L^2} \le \limsup_{\nu\to\infty}\|T_\nu(g_\nu)\|_{L^2} =\limsup_{\nu\to\infty} \|\frac {r_\nu^{1/2}f_\nu(r_\nu y, \sqrt{1-r_\nu^2 y^2})}{(1-|r_\nu y|^2)^{1/4}}\|_2=\limsup \|f_\nu\|_{L^2(\sigma,+)},$$
where the integral in $L^2(\sigma,+)$ should be understood as integrating over the upper hemisphere.

If $\mu(P^0)=0$, then we set $l=0$, and $e_\nu^0=g_\nu $ for all $\nu\ge 1$. Otherwise, $\mu(P^0)>0$, then up to a subsequence, there exists nontrivial $\phi^1\in L^2$ and $(x_\nu^1,t_\nu^1)_{\nu\ge 1}$ such that
\begin{align}
& \phi^1 =w-\lim_{\nu\to\infty}T_\nu^1(P_\nu^0)(y), \\
& \|\phi^1\|_{2}\ge \frac 12 \mu(P^0).
\end{align}Let $P^1$ denote the sequence $\{g_\nu(y)-(T_\nu^1)^{-1}(\phi^1)(y)\}_{\nu\ge 1}$ and set $$e_\nu^1 :=g_\nu(y)-(T_\nu^1)^{-1}(\phi^1)(y). $$ It is not hard to see that
\begin{align}
&\label{eq-b23} w-\lim_{\nu\to\infty} T_\nu^1(P_\nu^1)=0,\\
&\label{eq-b24} \|f_\nu\|_{L^2(S^1)}^2 -\|\phi^1\|_2^2= \|e_\nu^l\|_2^2, \text{ as }\nu\to\infty.
\end{align}

For $P^1=\{g_\nu(y)-(T_\nu^1)^{-1}(\phi^1)(y)\}_{\nu\ge 1}$, we iteratively consider the set
$$\mathcal{W}(P^1)=\{w-\lim_{\nu\to\infty} T_\nu (P_\nu^1)\text{ in } L^2(\R):\, (t_\nu,x_\nu)\in \R^2\}.$$
Then we test whether $\mu(P^1)>0$: if $\mu(P^1)=0$, then the algorithm stops.  If not, then up to a subsequence, there exists nontrivial $\phi^2\in L^2$ and and $(x_\nu^2,t_\nu^2)_{\nu\ge 1}$ such that
\begin{align}
&\phi^2 =w-\lim_{\nu\to\infty}T_\nu^2(P_\nu^1)(y), \\
& \|\phi^2\|_{2}\ge \frac 12 \mu(P^1).
\end{align}
By a similar consideration as in \eqref{eq-b23} and \eqref{eq-b24}, if setting $P^2= \{P_\nu^1-(T_\nu^2)^{-1}(\phi^2)\}$ and assuming \eqref{eq-b4}, then
\begin{align*}
 & w-\lim_{\nu\to\infty} T_\nu^2(P_\nu^2)=0,\\
 & \|f_\nu\|_2^2 -\sum_{j=1}^2 \|\phi^j\|_2^2=\|e_\nu^2\|_2^2, \text{ as }\nu\to\infty,
\end{align*}
where 
$$e_\nu^2 :=g_\nu-(T_\nu^1)^{-1}\phi^1-(T_\nu^2)^{-1}\phi^2.$$

The orthogonality in the $L^2$ norm above needs an input, namely, \eqref{eq-b4}. Otherwise, up to a subsequence we may assume that
$$|t_\nu^2-t_\nu^1|+|x_\nu^2-x_\nu^1|\to c, \text{ as }\nu\to\infty,$$for some $0\le c<\infty$. In this case, the dominated convergence theorem gives, up to a subsequence,
$$T_\nu^2(T_\nu^1)^{-1} \text{ converges strongly in } L^2.$$
This will imply that \begin{equation}\label{eq-b28}
T_\nu^2(P_\nu^1)\to 0, \text{ weakly in } L^2,
 \end{equation} as $T_\nu^1(P_\nu^1) \to 0$ weakly in $L^2$  and the following relation holds, $$T_\nu^2(P_\nu^1)=T_\nu^2(T_\nu^1)^{-1} \bigl(T_\nu^1(P_\nu^1)\bigr). $$
But the claim in \eqref{eq-b28} is a contradiction to the existence of nontrivial $\phi^2$.  So \eqref{eq-b4} holds.

Iterating this argument, a diagonalization process produces a family of pairwise orthogonal
sequences $(x_\nu^j, t_\nu^j)$ and $\phi^j$ satisfying \eqref{eq-b3}, \eqref{eq-b4} and \eqref{eq-L2-ortho}. Since $\sum_j \|\phi^j\|_2^2\le \sup_\nu \|f_\nu\|_2^2 <\infty$ and $\mu(P^{l+1}) \le 2\|\phi^l\|_2$, we have
\begin{equation}\label{eq-b7}
\mu(P^{l})\to 0, \text{ as } l\to\infty.
\end{equation}

To conclude this step, we deduce some information on  $e_\nu^l$. Firstly, the orthogonality condition \eqref{eq-b4} implies that, for any $\psi\in L^\infty$, the orthogonality \eqref{eq-b4} implies that, for each $l\ge 1$,
\begin{equation}\label{eq-c10} 
\|g_\nu \psi \|_{2}^2= \sum_{j=1}^l \|\phi^j \psi\|_2^2+ \|e_\nu^l \psi\|_2^2
\end{equation}
as $\nu\to \infty$. In particular, this holds for $\psi \in \mathcal{S}$, the Schwartz class on $\mathbb{R}$. 

Let $R\gg1$. Define a set 
$$E= \{y\in \mathbb{R}:\, |y|\le R \text{ and } |g_\nu(y)|\le R \text{ for all sufficiently large }\nu\}. $$
Then \eqref{eq-c10} implies that, for any $l\ge 1$,
$$\limsup_{\nu\to \infty} \|e^l_\nu 1_E\|_{L^\infty} \le C R$$
for some $C>0$. This further implies that,
$$\limsup_{\nu\to \infty} \|(1-r_\nu^2 y^2)^{1/4}e^l_\nu 1_E\|_{L^\infty} \le C R. $$

\textbf{Step 2.} At this step, we show that the localized restriction estimate $L^\infty\to L^q_{t,x}$ for some $q<6$ in Lemma \ref{le-local-restr}, together with the information that $\lim \mu(P^l)=0$, will imply \eqref{eq-b8}. To do it, by scaling, the norm on the left hand side of \eqref{eq-b8} is equivalent to 
$$\|\int e^{ixy+it\frac {\sqrt{1-r_\nu^2 y^2}-1}{r_\nu^2}}\frac {e_\nu^l(y)}{(1-r_\nu^2y^2)^{1/4}}dy\|_{L^6_{t,x}}.$$ 

For each $R\gg 1$, recall the definition of the set $E$. We split
$$e_\nu^l = e_\nu^l 1_E + e_\nu^l(1-1_E). $$
Since the following operator is uniformly bounded from $L^2(\mathbb{R})$ to $L^6_{t,x}(\mathbb{R}\times \mathbb{R})$, 
$$ \phi \mapsto \int e^{ixy+it\frac {\sqrt{1-r_\nu^2 y^2}-1}{r_\nu^2}}\frac {\phi(y)}{(1-r_\nu^2y^2)^{1/4}}dy,$$
and $f_\nu$ is upper normalized with respect to $\mathcal{C}(z_\nu, r_\nu)$, up to a subsequence, we may conclude that by \eqref{eq-c10}
$$ \|\int e^{ixy+it\frac {\sqrt{1-r_\nu^2 y^2}-1}{r_\nu^2}}\frac {e_\nu^l(y)(1-1_E)}{(1-r_\nu^2y^2)^{1/4}}dy \|_{L^6_{t,x}}\le C \|e_\nu^l(y)(1-1_E) \|_2 \le C \|g_\nu(1-1_E) \|_2\le \Theta(R)$$
as $\nu \to \infty$. So we may restrict our attention to $e_\nu^l$ on $E$. By the discussion at the end of \textbf{Step 1}, we may assume that, for all $l\ge 1$, 
\begin{equation}\label{eq-33}
\limsup_{\nu\to \infty}\|(1-r_\nu^2y^2)^{\frac 14}e_\nu^l 1_E \|_\infty \le  CR. 
\end{equation}  Then by Lemma \ref{le-local-restr}, 
$$\limsup_{\nu\to \infty} \|e^{\frac {it\Delta}{2}}\bigl(h_\nu(t)\frac {e_\nu^l1_E}{(1-r_\nu^2y^2)^{1/4}}\bigr)\|_{L^q_{t,x}} \le C$$
for some $C>0$ independent of $\nu$ and $l$. Then by the interpolation, establishing \eqref{eq-b8} is reduced to
\begin{equation}\label{eq-b11}
\limsup_{l\to\infty}\limsup_{\nu\to\infty} \|e^{\frac {it\Delta}{2}}\bigl(h_\nu(t)\frac {e_\nu^l(y)}{(1-r_\nu^2y^2)^{1/4}}\bigr)\|_{L^\infty_{t,x}} =0.
\end{equation}
This will follow from the fact that $\mu(P^l) \to 0$ as $l\to\infty.$ Indeed, there exists $(x^l_\nu,t^l_\nu)$ such that, up to a subsequence,
\begin{equation}\label{eq-b12}
\left|e^{\frac {it_\nu^l\Delta}{2}}\bigl(h_\nu(t_\nu^l))\frac {e_\nu^l1_E}{(1-r_\nu^2y^2)^{1/4}} \bigr)(x_\nu^l)\right|\sim\|e^{\frac {it\Delta}{2}}\bigl(h_\nu(t))\frac {e_\nu^l1_E}{(1-r_\nu^2y^2)^{1/4}}\bigr)\|_{L^\infty_{t,x}}.
\end{equation}
On the other hand, since $e_\nu^l$ is compactly supported, 
$$ e^{-\frac {it^l_\nu y^2}{2}}e^{ix^l_\nu y}h_\nu(t_\nu^l,y) )\frac {e_\nu^l1_E}{(1-r_\nu^2y^2)^{1/4}} =e^{-\frac {it^l_\nu y^2}{2}}e^{ix^l_\nu y}h_\nu(t_\nu^l,y) )\frac {e_\nu^l1_E}{(1-r_\nu^2y^2)^{1/4}} \phi_R(y)$$ 
for some suitable bump function $\phi_R$ adapted to the ball $B(0,R)$; taking integration in $y$ on both sides, we have
\begin{equation}\label{eq-b14}
e^{\frac {it_\nu^l\Delta}{2}}\bigl(h_\nu(t_\nu^l,y))\frac {e_\nu^l1_E}{(1-r_\nu^2y^2)^{1/4}} \bigr) (x_\nu^l)
=\langle e^{-\frac {it^l_\nu y^2}{2}}e^{ix^l_\nu y}h_\nu(t_\nu^l,y)e_\nu^l,\, \frac {\phi_R1_E}{(1-r_\nu^2y^2)^{1/4}}\rangle_{L_y^2}.
\end{equation}
Since $P^l=\{e_\nu^l\}_{\nu\ge 1}$, by the definition of $\mu(P^l)$,
\begin{equation}
\text{ LHS }\eqref{eq-b12} \le \mu(P^l) \|\frac {\phi_R1_E}{(1-r_\nu^2y^2)^{1/4}}\|_{L^2}\le \mu(P^l) \|\phi_R1_E\|_{L^2} \to 0, \text{ as } l\to\infty,
\end{equation} since $\mu(P^l)\to 0$ as $l\to\infty.$ This finishes the proof of \eqref{eq-b8}.

Therefore the proof of Proposition \ref{prop-decomp} is complete.
\end{proof}

Next we show that \eqref{eq-b4} implies the orthogonality result \eqref{eq-b6} in Proposition \ref{prop-ortho}.
\begin{proof}[The proof of Proposition \ref{prop-ortho}.] To begin, we may assume that $\phi^j$ and $\phi^k$ are smooth functions with compact supports. Also we recall that
$$ e^{\frac {it\Delta}{2}}\bigl( h_\nu(t,y)\frac {G_\nu^k}{(1-r_\nu^2y^2)^{1/4}}\bigr)=\int e^{i(x-x_\nu^k)-\frac {i(t-t_\nu^k)y^2}{2}}e^{i(t-t_\nu^k)\bigl(\frac {\sqrt{1-r^2_\nu y^2}-1}{r_\nu^2}+\frac {y^2}{2}\bigr)} \frac {\phi^k(y)}{(1-r_\nu^2y^2)^{1/4}} dy. $$
Likewise for $e^{\frac {it\Delta}{2}}\bigl( h_\nu(t,y)\frac {G_\nu^j}{(1-r_\nu^2y^2)^{1/4}}\bigr)$. Then by a change of variables, we need to show
\begin{equation}\label{eq-b15}
\begin{split}
&\left\| e^{i\frac {t-(t_\nu^j-t_\nu^k)}{2}\Delta}\Bigl( e^{i\bigl(t-(t_\nu^j-t_\nu^k)\bigr)\bigl(
\frac {\sqrt{1-r_\nu^2 y^2}-1}{r_\nu^2}+\frac {y^2}{2}\bigr)}\frac {\phi^j}{(1-r_\nu^2 y^2)^{1/4}}\Bigr)\bigl(x-(x_\nu^j-x_\nu^k)\bigr)\times \right. \\
&\qquad\qquad \times \left.e^{\frac {it\Delta}{2}}\Bigl( e^{it\bigl(
\frac {\sqrt{1-r_\nu^2 y^2}-1}{r_\nu^2}+\frac {y^2}{2}\bigr)}\frac {\phi^k}{(1-r_\nu^2y^2)^{1/4}}\Bigr)\right\|_{L^3_{t,x}} \to 0
\end{split}
\end{equation} as $\nu$ goes to infinity.

For a large $N\gg 1$, set
$$\Omega_N:=\{(t,x): |t|+|x|\le N\},\quad \Omega_{N,\nu}=\Omega_N-(t_\nu^j-t_\nu^k, x_\nu^j-x_\nu^k).$$
We first claim that, for $\Omega=\Omega_N$ or $\Omega_{N,\nu}$,
\begin{equation}\label{eq-b16}
\begin{split}
&\int_{\Omega^c} \left|  e^{i\frac {t-(t_\nu^j-t_\nu^k)}{2}\Delta}\Bigl( e^{i\bigl(t-(t_\nu^j-t_\nu^k)\bigr)\bigl(
\frac {\sqrt{1-r_\nu^2 y^2}-1}{r_\nu^2}+\frac {y^2}{2}\bigr)}\frac {\phi^j}{(1-r_\nu^2y^2)^{1/4}}\Bigr) \bigl(x-(x_\nu^j-x_\nu^k)\bigr) \times \right. \\
&\qquad \qquad \times \left. e^{\frac {it\Delta}{2}}\Bigl( e^{it\bigl(
\frac {\sqrt{1-r_\nu^2 y^2}-1}{r_\nu^2}+\frac {y^2}{2}\bigr)}\frac {\phi^k}{(1-r_\nu^2y^2)^{1/4}}\Bigr)\right|^3dxdt\to 0
\end{split}
\end{equation} as $N$ goes to infinity uniformly in $\nu$. Here $\Omega^c:=\R^2\setminus \Omega.$

We write
$$e^{\frac {it\Delta}{2}}\Bigl( e^{it\bigl(\frac {\sqrt{1-r_\nu^2 y^2}-1}{r_\nu^2}+\frac {y^2}{2}\bigr)}\frac {\phi^k}{(1-r_\nu^2 y^2)^{1/4}}\Bigr)(x)=\int e^{ixy+it\frac {\sqrt{1-|r_\nu y|^2}-1}{r_\nu^2}} \frac {\phi^k(y)}{(1-r_\nu^2y^2)^{1/4}}dy.$$
For $y$ in a compact set in $\mathbb{R}$ and all sufficiently small $r_\nu>0$, we have
\begin{equation}\label{eq-34}
\left|\partial_y^2 \bigl(\frac {\sqrt{1-|r_\nu y|^2}-1}{r_\nu^2} \bigr) \right|\sim 1/4, \text{ uniformly in all sufficiently large }\nu.
\end{equation}
We state three important estimates uniformly in all sufficiently large $\nu$. The first is by the stationary phase estimate \cite[p.334]{Stein:1993}:
\begin{equation}\label{eq-b17}
\left|e^{\frac {it\Delta}{2}}\Bigl( e^{it\bigl(\frac {\sqrt{1-r_\nu^2 y^2}-1}{r_\nu^2}+\frac {y^2}{2}\bigr)}\frac {\phi^k}{(1-r_\nu^2 y^2)^{1/4}} \Bigr)(x)\right|\le C_{\phi^k} |t|^{-1/2}.
\end{equation} Secondly by integration by parts, if $|x|\ge C|t|$ for a large constant $C>0$ depending on the size of the compact support of $\phi^k$, for all sufficiently large $\nu$,
\begin{equation}\label{eq-b18}
\left|e^{\frac {it\Delta}{2}}\Bigl( e^{it\bigl(\frac {\sqrt{1-r_\nu^2 y^2}-1}{r_\nu^2}+\frac {y^2}{2}\bigr)}\frac {\phi^k}{(1-r_\nu^2 y^2)^{1/4}} \Bigr)(x)\right|\le C_{\phi^k} |x|^{-1}.
\end{equation} Thirdly there always holds a trivial bound, for all $x,t$,
\begin{equation}\label{eq-b19}
\left|e^{\frac {it\Delta}{2}}\Bigl( e^{it\bigl(\frac {\sqrt{1-r_\nu^2 y^2}-1}{r_\nu^2}+\frac {y^2}{2}\bigr)}\frac {\phi^k}{(1-r_\nu^2 y^2)^{1/4}} \Bigr)(x)\right|\le C_{\phi^k}.
\end{equation}
Here all constants $C_{\phi^k}$ depends on the function $\phi^k$ but independent of $\nu$. We are now ready to prove \eqref{eq-b16} when $\Omega=\Omega_N$; the case where $\Omega=\Omega_{N,\nu}$ is similar and so will be omitted. By the Cauchy-Schwarz inequality,
\begin{equation}
\begin{split}
\text{ LHS of }\eqref{eq-b16} &\le C \left\| e^{i\frac {t-(t_\nu^j-t_\nu^k)}{2}\Delta}\Bigl( e^{i\bigl(t-(t_\nu^j-t_\nu^k)\bigr)\bigl(
\frac {\sqrt{1-r_\nu^2 y^2}-1}{r_\nu^2}+\frac {y^2}{2}\bigr)}\frac {\phi^j}{(1-r_\nu^2 y^2)^{1/4}}\Bigr) \bigl(x-(x_\nu^j-x_\nu^k)\bigr)\right\|^3_{L^6(\R^2)} \\
&\qquad \times \left\|e^{\frac {it\Delta}{2}}\Bigl( e^{it\bigl(
\frac {\sqrt{1-r_\nu^2 y^2}-1}{r_\nu^2}+\frac {y^2}{2}\bigr)}\frac {\phi^k}{(1-r_\nu^2 y^2)^{1/4}}\Bigr)\right\|^3_{L^6_{t,x}(\Omega^c_N)}\\
&\le C \left\| e^{\frac {it\Delta}{2}}\Bigl( e^{it\bigl(\frac {\sqrt{1-r_\nu^2 y^2}-1}{r_\nu^2}+\frac {y^2}{2}\bigr)}\frac {\phi^j}{(1-r_\nu^2 y^2)^{1/4}}\Bigr)\right\|^3_{L^6(\R^2)} \times\\
&\qquad \qquad \times \left\|e^{\frac {it\Delta}{2}}\Bigl( e^{it\bigl(
\frac {\sqrt{1-r_\nu^2 y^2}-1}{r_\nu^2}+\frac {y^2}{2}\bigr)}\frac {\phi^k}{(1-r_\nu^2 y^2)^{1/4}}\Bigr)\right\|^3_{L^6_{t,x}(\Omega^c_N)}.
\end{split}
\end{equation}
The first term is bounded by the Tomas-Stein inequality and a change of variables. For the second term, by using estimates \eqref{eq-b17}, \eqref{eq-b18} and \eqref{eq-b19}, we see that
\begin{equation}\label{eq-b21}
\left\|e^{\frac {it\Delta}{2}}\Bigl( e^{it\bigl(
\frac {\sqrt{1-r_\nu^2 y^2}-1}{r_\nu^2}+\frac {y^2}{2}\bigr)}\frac {\phi^k}{(1-r_\nu^2 y^2)^{1/4}}\Bigr)\right\|^3_{L^6_{t,x}(\Omega^c_N)}\to 0, \text{ as }  N\to \infty, \text{ uniform in } \nu.
\end{equation}
Therefore we have established \eqref{eq-b16}. To finish the proof \eqref{eq-b15}, we need to show that, for a fixed $N\gg 1$,
\begin{equation}\label{eq-b22}
\begin{split}
& \int_{\Omega_N\cap\Omega_{N,\nu}} \left|e^{i\frac {t-(t_\nu^j-t_\nu^k)}{2}\Delta}\Bigl( e^{i\bigl(t-(t_\nu^j-t_\nu^k)\bigr)\bigl(
\frac {\sqrt{1-r_\nu^2 y^2}-1}{r_\nu^2}+\frac {y^2}{2}\bigr)}\frac {\phi^j}{(1-r_\nu^2 y^2)^{1/4}}\Bigr)\bigl(x-(x_\nu^j-x_\nu^k)\bigr) \times \right. \\
&\qquad \qquad \times \left. e^{\frac {it\Delta}{2}}\Bigl( e^{it\bigl(
\frac {\sqrt{1-r_\nu^2 y^2}-1}{r_\nu^2}+\frac {y^2}{2}\bigr)}\frac {\phi^k}{(1-r_\nu^2 y^2)^{1/4}}\Bigr)\right|^3 \to 0
\end{split}
\end{equation} as $\nu$ goes to infinity. It actually holds as
$$\operatorname{measure}(\Omega_N\cap \Omega_{N,\nu}) \to 0, \text{ when } \lim_{\nu\to\infty} |t_\nu^j-t_\nu^k|+|x_\nu^j-x_\nu^k|=\infty,$$
and we can apply $L^\infty_{t,x}$-bounds to both integrals, which are controlled as $\phi^j$ and $\phi^k$ are assumed to be bounded and compactly supported. Therefore the proof of \eqref{eq-b6} is complete.
\end{proof}

\subsection{Ruling out small caps}\label{sec-rule-out-small}
By the discussion at the beginning of Section \ref{sec:step5}, we aim to show that
\begin{equation}\label{eq-b32}
\lim_{\nu\to\infty} \|\widehat{f^+_\nu \sigma }\|_6^6 \le \mathcal{R}^6_{\textbf{P}},
 \end{equation}
which leads to $\mathcal{R}\le (5/2)^{1/6}\mathcal{R}_{\textbf{P}}$. However, it is a contradiction to the strict inequality in Proposition \ref{prop-strict-inequ}.

By Propositions \ref{prop-decomp} and \ref{prop-ortho}, 
\begin{equation}\label{eq-b31}
\begin{split}
\lim_{\nu\to\infty} \|\widehat{f^+_\nu \sigma }\|_6^6 
&\le \sum_{j=1}^\infty \lim_{\nu\to\infty} \|e^{\frac {it\Delta}{2}}\bigl(h_\nu(t,y)\frac {G_\nu^j}{(1-r_\nu^2y^2)^{1/4}}\bigr)\|_6^6\\
&=\sum_{j}\lim_{\nu\to\infty}\left\|\int e^{i(x-x_\nu^j)y-\frac {(t-t_\nu^j)y^2}{2}}e^{i(t-t_\nu^j)\bigl(\frac {\sqrt{1-r_\nu^2 y^2}-1}{r_\nu^2}+\frac {y^2}{2}\bigr)}\frac {\phi^j(y)}{(1-r_\nu^2y^2)^{1/4}}dy \right\|_6^6\\
&=\sum_{j}\lim_{\nu\to\infty}\left\|\int e^{ixy-\frac {ty^2}{2}}e^{it\bigl(\frac {\sqrt{1-r_\nu^2 y^2}-1}{r_\nu^2}+\frac {y^2}{2}\bigr)}\frac {\phi^j(y)}{(1-r_\nu^2y^2)^{1/4}} dy \right\|_6^6\\
&=\sum_j \|e^{\frac {it\Delta}{2}}\phi^j\|_6^6\le \mathcal{R}^6_{\textbf{P}}\sum_j \|\phi^j\|_2^6\\
&\le \mathcal{R}^6_{\textbf{P}} \bigl(\sum_j \|\phi^j\|^2_2\bigr)^3\\
&\le \mathcal{R}^6_{\textbf{P}} \lim_{\nu\to\infty} \|f^+_\nu\|_2^6=\mathcal{R}^6_{\textbf{P}}.
\end{split}
\end{equation} This proves \eqref{eq-b32}. Here we have used
$$\lim_{\nu\to\infty}\left\|\int e^{ixy-\frac {ty^2}{2}}e^{it\bigl(\frac {\sqrt{1-r_\nu^2 y^2}-1}{r_\nu^2}+\frac {y^2}{2}\bigr)}\frac {\phi^j(y)}{(1-r_\nu^2y^2)^{1/4}} dy -e^{\frac {it\Delta}{2}}\phi^j(x)\right\|_6 =0.$$
This follows from the stationary phase analysis and the dominated convergence theorem. So far we have proved that the first half of Proposition \ref{prop-regularity-after-rescaling} that $\inf_\nu r_\nu>0$.

\subsection{Big caps; existence of extremals.}\label{sec-big-caps}
In this section we aim to prove the second half of Proposition \ref{prop-regularity-after-rescaling}: There exists an extremal function for the Tomas-Stein inequality \eqref{eq-1}. The proof is similar to the process of ruling out small caps above. Let $\{f_\nu\}$ be an extremizing sequence of nonnegative functions supported on the whole sphere and even upper normalized with respect to caps $\mathcal{C}_\nu \cup (-\mathcal{C}_\nu)$. We have proved that $\inf_\nu r_\nu>0.$ Then up to a subsequence, the uniform upper normalization means simply that $\|f_\nu\|_{L^2(S^1)} \le 1$, and there exists a function $\Theta$ independent of $\nu$ and satisfying that $\Theta(R)\to 0$ as $R\to \infty$, such that 
$$ \int_{|f_\nu (x)|>R} |f_\nu(x)|^2d\sigma (x) \le \Theta(R)$$
for all $\nu$. The radii no longer enter into the discussion.

We denote $f_\nu^\pm$ the restrictions of $f_\nu$ to the upper hemisphere $S^1_+$ and the lower. Then we see that $f_\nu^+(x)=f_\nu^-(-x)$, for $x\in S^1_+$, and $\|f_\nu\|_2^2=2\|f_\nu^+\|_2^2$. Moreover, by a simple change of variables,
\begin{equation}\label{eq-37}
\widehat{f_\nu^-\sigma} (t,x)=\widehat{f_\nu^+\sigma}(-t,-x)=\overline{\widehat{f_\nu^+\sigma}}(t,x).
\end{equation} Then
\begin{equation}\label{eq-38}
\widehat{f_\nu\sigma}(t,x) =\widehat{f_\nu^+\sigma} (t,x)+\widehat{f_\nu^-\sigma} (t,x)=2\Re \widehat{f_\nu^+\sigma} (t,x),
\end{equation} where $\Re f$ denotes the real part of $f$. 

Write 
\begin{align*}
\widehat{f_\nu^+}(x,t) &= \int_{S^1_+} e^{i(x,t) \cdot z} f_\nu^+(z) d\sigma(z) \\
 &= \int e^{ixy+it\sqrt{1-y^2}}f_\nu^+(y,\sqrt{1-y^2}) \frac {dy}{\sqrt{1-y^2}}.
\end{align*}

Similar to Propositions \ref{prop-decomp} and \ref{prop-ortho} in the subsection \ref{sec-key-prop}, we will develop a profile decomposition for ${f_\nu^+(y,\sqrt{1-y^2})}/{(1-y^2)^{1/4}}$. 

\begin{proposition}\label{prop-decomp-bigcaps}
Let $\{f^+_\nu\}$ be defined above. Then there exists a sequence $(x_\nu^k,t_\nu^k)\in \R^2$ and $e_\nu^l\in L^2(\R)$ such that
\begin{equation}\label{eq-c3}
\frac {f^+_\nu(y)}{(1-y^2)^{1/4}}=\sum_{j=1}^l e^{-ix^j_\nu y-it_\nu^j\sqrt{1-y^2}} \phi^j(y)+e^l_\nu(y)
\end{equation} with the following properties: The parameters $\{(x_\nu^k,t_\nu^k)\}$ satisfy, for $k\neq j$,
\begin{equation}\label{eq-c4}
|x_\nu^k-x_\nu^j|+|t_\nu^k-t_\nu^j|\to \infty, \text{ as }\nu\to\infty.
\end{equation}
For each $l\ge 1$,
\begin{equation}
\|f^+_\nu\|_{L^2(S^1)}^2= \sum_{j=1}^l\|\phi^j\|_2^2+\|e_\nu^l\|_2^2, \text{ as } \nu\to \infty.
\end{equation}
The function $e_\nu^l$ satisfies, if $E_\nu^l=(1-y^2)^{1/4}e_\nu^l$
\begin{equation}\label{eq-c5}
\limsup_{l\to\infty}\limsup_{\nu\to\infty}\left\| \widehat{E_\nu^l \sigma}\right\|_{L^6_{t,x}(\R^2)}=0.
\end{equation}
\end{proposition}

\begin{proposition}[Orthogonality]\label{prop-ortho-bigcaps}
Let $\{(x_\nu^j,t_\nu^j)\}$ be as above and set
\begin{align}
G_\nu^k& :=e^{-ix_\nu^k y - it_\nu^k \sqrt{1-y^2}} (1-y^2)^{\frac 14}\phi^k,\\
G_\nu^j& :=e^{-ix_\nu^j y - it_\nu^j \sqrt{1-y^2}} (1-y^2)^{\frac 14} \phi^j.
\end{align} Then for $k\neq j$,
\begin{equation}\label{eq-c6}
\begin{split}
&\lim_{\nu\to\infty} \left\|\widehat{G_\nu^k\sigma} \widehat{G_\nu^j\sigma} \right\|_{L^3_{t,x}(\R^2)}=0.
\end{split}
\end{equation}
\end{proposition}
The proofs are similar and so we omit the details. Now we are ready to prove the existence of extremals for \eqref{eq-1}. 
\begin{equation}\label{eq-b33}
\begin{split}
& \mathcal{R}^6=\lim_{\nu\to\infty} \|\widehat{f_\nu \sigma }\|_6^6  =2^6 \limsup_{l\to\infty}  \lim_{\nu\to\infty} \|\Re \left\{ \sum_{j=1}^l \widehat{G_\nu^j\sigma}+\widehat{E_\nu^l\sigma}\right\}\|_6^6\\
&\le 2^6 \limsup_{l\to\infty}  \lim_{\nu\to\infty} \|\Re \sum_{j=1}^l \widehat{G_\nu^j\sigma} \|_6^6\\
&=\sum_{j}\left\|\int e^{i(x,t)\cdot z} (1-y^2)^{1/4}[\phi^j(y)1_{S^1_+}(z)+ \bar{\phi^j}(y)1_{-S^1_+}(z)]d\sigma(y) \right\|_6^6\\
&\le \mathcal{R}^6 \sum_j  \|(1-y^2)^{1/4}[\phi^j(y)1_{S^1_+}(z)+ \bar{\phi^j}(y)1_{-S^1_+}(z)] \|_{L^2(S^1,\sigma)}^6 \\
&\le \mathcal{R}^6 \sum_j  \left(2 \|\phi^j \|_2^2\right)^3 \le \mathcal{R}^6 \left(\sum_j 2\|\phi^j\|^2_2\right)^3\\
&\le \mathcal{R}^6 \left(2 \|f_\nu^+\|^2_2\right)^3= \mathcal{R}^6 \|f_\nu\|_2^6 = \mathcal{R}^6.
\end{split}
\end{equation}
where $z=(y,\cdot)\in S^1$, and $1_{S^1_+}$ and $1_{-S^1_+}$ denotes the indicator functions of the upper and the lower hemispheres of $S^1$, respectively.

Then $\mathcal{R}^6=\mathcal{R}^6$ forces all the inequalities above to be equal. On the other hand, because $2 \sum_j \|\phi^j\|^2_{L^2}\le \|f_\nu\|_{L^2(S^1,\sigma)}=1$, there will be only one $j$ left from the sharpness of embedding of $\ell^3$ into $\ell^1$. Thus for this $j$, there exists an extremal $$\left(\phi(y)1_{S^1_+}+\bar{\phi^j}(y)1_{-S^1_+}\right)(1-y^2)^{1/4}.$$
This completes the proof of Proposition  \ref{prop-regularity-after-rescaling} and hence Theorem \ref{thm-main}.

\appendix
\section{A strict comparison, $\mathbf{S}>(5/2)^{1/6}\mathbf{P}$.}\label{sec:strict-comparison}
In this section, we aim to establish the strict comparison inequality in Proposition \ref{prop-strict-inequ} by using a similar perturbation argument in \cite[Section 17]{Christ-Shao:extremal-for-sphere-restriction-I-existence} on $$\|\widehat{f_\eps\sigma}\|_6^6/\|f_\eps\|_2^6,$$ where $f_\eps$ is defined in \eqref{eq-76}. 

We list several definitions. 
\begin{equation}
\bigl(y,(1-|y|^2)^{1/2}\bigr) =\bigl(y,1-\frac 12 |y|^2-\frac 18|y|^4+O(|y|^6)\bigr),
\end{equation}
\begin{equation}
d\sigma_\eps: =\bigl(1+\frac 12|y|^2+O(|y|^4)\bigr) dy,
\end{equation}
\begin{equation}
\label{eq-76} f_\eps (z):= \eps^{-1/4} e^{(z_2-1)/\eps}\chi_{|z_1|\le \frac 12}\chi_{z_2>0}.
\end{equation}
\begin{equation}
\begin{split}
&u_\eps (t,x):=\int_{S^1} f_\eps(z)e^{-i(x,t)\cdot z} d\sigma(z)\\
&\qquad \qquad =\eps^{-1/4} e^{-it}\int_{\R} e^{\bigl(-\frac 12|y|^2-\frac 18|y|^4+O(|y|^6)\bigr)\eps^{-1}}\times \\
&\qquad \qquad \qquad \times e^{-ix\cdot y} e^{-it\bigl(-\frac 12|y|^2-\frac 18|y|^4+O(|y|^6)\bigr)}
\bigl(1+\frac 12|y|^2+O(|y|^4)\bigr)\chi_{|y|\le 1/2}(y)dy\\
&\qquad \qquad =\eps^{1/4}e^{-it} \int_\R e^{-i\eps^{1/2}x\cdot y}e^{-(1-i\eps t)\bigl(\frac 12|y|^2+\frac \eps 8|y|^4+O(\eps^2|y|^6)\bigr)}\times \\
&\qquad \qquad \qquad \times \bigl(1+\frac \eps 2|y|^2+O(\eps^2|y|^4)\bigr)\chi_{|y|\le 1/2}(\eps^{1/2}y)dy,
\end{split}
\end{equation}where a change of variables is applied in passing to the last inequality. We continue to set
\begin{equation}
\begin{split}
v_\eps (t,x)&:=\eps^{-1/4}u_\eps (\eps^{-1}t,\eps^{-1/2}x)\\
&= e^{-i\eps^{-1}t} \int_\R e^{-ix\cdot y}e^{-(1-it)\bigl(\frac 12|y|^2+\frac\eps 8|y|^4+O(\eps^2|y|^6)\bigr)}\times \\
&\qquad \qquad \times \bigl(1+\frac \eps 2|y|^2+O(\eps^2|y|^4)\bigr)\chi(\eps^{1/2}y)dy,
\end{split}
\end{equation}
\begin{equation}
w_\eps(t,x) :=\int_\R e^{-ix\cdot y}e^{-(1+it)\bigl(\frac 12|y|^2+\frac\eps 8|y|^4\bigr)}\bigl(1+\frac \eps 2|y|^2\bigr)dy.
\end{equation}
\begin{equation}
g_\eps(y):=e^{-\frac 12|y|^2-\frac {\eps}{8}|y|^4},
\end{equation}
\begin{equation}
d\sigma_\eps(y)=\bigl(1+\frac\eps 2|y|^2\bigr)dy.
\end{equation}Note that $1-it\to 1+it$ when passing $v_\eps$ to $w_\eps$, which amounts to a complex conjugation. Then we see that
\begin{equation}
\|v_\eps\|^6_{L^6(\R^2)}=\|u_\eps\|^6_{L^6(\R^2)}, 
\end{equation}
\begin{equation}
\begin{split}
\|w_\eps\|^6_{L^6(\R^2)}&=\|v_\eps\|^6_{L^6(\R^2)}+ O(\eps^2)\\
&=\|u_\eps\|^6_{L^6(\R^2)}+O(\eps^2),\text{ as }\eps\to 0^+.
\end{split}
\end{equation}
\begin{equation}
\|f_\eps\|_{L^2(S^1,\sigma)}^2=\|g_\eps\|_{L^2(\mathbb{R},\sigma_\eps)}^2+ O(\eps^2),\text{ as }\eps\to 0^+.
\end{equation}
We consider the functional
\begin{equation}\label{eq-20}
\Psi(\eps)=\log \frac {\|u_\eps\|_6^6}{\|f_\eps\|_2^6},
\end{equation} which is initially defined for $\eps >0$ and extends continuously and differentially to $\eps=0$. The derivative is
\begin{equation}\label{eq-21}
\partial_\eps|_{\eps=0} \Psi(\eps) =\frac {\partial_\eps|_{\eps=0} \|w_\eps\|_6^6}{\|w_0\|_6^6}
-3\frac {\partial_\eps|_{\eps=0} \|g_\eps\|_2^2}{\|g_0\|_2^2}.
\end{equation} We observe that,
\begin{equation}
\Psi(0)=\log (\mathcal{R}_P^6).
\end{equation}
We will calculate that
\begin{lemma}\label{le-4}
\begin{equation}\label{eq-22}
\partial_\eps|_{\eps=0} \Psi(\eps)>0.
\end{equation}
\end{lemma}
\begin{proof}
\begin{equation}\label{eq-24}
\begin{split}
\partial_\eps|_{\eps=0}w_\eps&=\bigl[\frac 12 (1+it)\partial_t^2+i\partial_t\bigr]\int_\R e^{-ix\cdot y-\frac {1+it}{2}|y|^2}dy\\
&=\bigl[\frac 12 (1+it)\partial_t^2+i\partial_t\bigr] c_0(1+it)^{-1/2}e^{-\frac {|x|^2}{2(1+it)}}\\
&=\bigl[\frac 12 (1+it)\partial_t^2+i\partial_t\bigr]w_0(t,x),
\end{split}
\end{equation}where $c_0>0$ is some universal constant and $w_0:=c_0(1+it)^{-1/2}e^{-\frac {|x|^2}{2(1+it)}}$.
Define \begin{equation}
\phi(t,x):=-\frac 12 |x|^2(1+it)^{-1}-\frac 12 \log(1+it).
\end{equation} Then
\begin{equation}
w_0(t,x)=c_0(1+it)^{-1/2}e^{-\frac {|x|^2}{2(1+it)}}=c_0e^{\phi}.
\end{equation}Continuing the computation in \eqref{eq-24},
\begin{equation}
\bigl[\frac {1+it}{2}(\phi_t^2+\phi_{tt})+i\phi_t\bigr]w_0.
\end{equation}
We compute
\begin{align}
\phi_t&=\frac i2|x|^2(1+it)^{-2}-\frac i2(1+it)^{-1},\\
\phi_{tt}&=|x|^2(1+it)^{-3}-\frac 12(1+it)^{-2},\\
\phi_t^2&=-\frac 14|x|^4(1+it)^{-4}+\frac 12 |x|^2(1+it)^{-3}-\frac 14(1+it)^{-2}.
\end{align} Thus
\begin{equation}
\phi^2_{t}+\phi_{tt}=-\frac 14|x|^4(1+it)^{-4}+\frac 32|x|^2(1+it)^{-3}-\frac 34(1+it)^{-2}.
\end{equation}Then
\begin{equation}\label{eq-25}
\begin{split}
&\frac {1+it}{2}(\phi_t^2+\phi_{tt})+i\phi_t=\\
&=-\frac 18|x|^4(1+it)^{-3}+\frac 14|x|^2(1+it)^{-2}+\frac 18(1+it)^{-1}.
\end{split}
\end{equation}
Taking the real part in \eqref{eq-25}, we have
\begin{equation}
\begin{split}
&\Re \bigl[\frac {1+it}{2}(\phi_t^2+\phi_{tt})+i\phi_t\bigr]\\
&=-\frac 18|x|^4(1+t^2)^{-3}(1-3t^2)+\frac 14|x|^2(1+t^2)^{-2}(1-t^2)+\frac 18(1+t^2)^{-1}.
\end{split}
\end{equation}
Since
\begin{equation}
\partial_\eps \|w_\eps\|_6^6 =\partial_\eps \int |w_\eps|^6=\partial_\eps \int \bigl(w_\eps\overline{w}_\eps\bigr)^3\\
=6\int |w_\eps|^6 \Re \left(\frac {\partial_\eps w_\eps}{w_\eps}\right),
\end{equation} we have
\begin{equation}\label{eq-26}
\begin{split}
&\partial_\eps|_{\eps=0} \|w_\eps\|_6^6= 6\iint \Re \bigl[\frac {1+it}{2}(\phi_t^2+\phi_{tt})+i\phi_t\bigr]|w_0|^6dxdt\\
&=c_0^6 \iint \bigl[ -\frac 34|x|^4(1+t^2)^{-3}(1-3t^2)+\frac 32|x|^2(1+t^2)^{-2}(1-t^2)+\frac 34(1+t^2)^{-1}\bigr]\times \\
&\qquad \qquad \qquad\times (1+t^2)^{-3/2}e^{-3\frac {|x|^2}{1+t^2}}dxdt\\
&=c_0^6\iint \bigl[ -\frac 34|x|^4(1-3t^2)+\frac 32|x|^2(1-t^2)+\frac 34\bigr](1+t^2)^{-2}e^{-3|x|^2}dxdt\\
&=c_0^6\int \left(-\frac 34(1-3t^2)\frac {\sqrt\pi}{12\sqrt 3}+\frac 32(1-t^2)\frac {\sqrt\pi}{6\sqrt3}+\frac {3\sqrt\pi}{4\sqrt3}\right)(1+t^2)^{-2}dt,
\end{split}
\end{equation}where we have used that
\begin{equation*}
\int_\R e^{-3|x|^2}dx =\frac {\sqrt\pi}{\sqrt3},\,\int_\R |x|^2 e^{-3|x|^2}dx =\frac {\sqrt\pi}{6\sqrt 3},\text{ and } \int_\R |x|^4e^{-3|x|^2}dx =\frac {\sqrt\pi}{12\sqrt3}.
\end{equation*}
Hence we continue \eqref{eq-26},
\begin{equation}
\begin{split}
=&c_0^6\left(-\frac {\sqrt\pi}{16\sqrt3}\int_\R (t^2-3)(1+t^2)^{-2}dt+\frac {3\sqrt\pi}{4\sqrt3} \int_\R(1+t^2)^{-2}dt\right)\\
&=c_0^6\frac {\sqrt\pi}{16\sqrt3}\left(-\int_\R t^2(1+t^2)^{-2}dt+15\int_\R (1+t^2)^{-2}dt\right)\\
&=c_0^6\frac {7\pi\sqrt\pi}{16\sqrt3}
\end{split}
\end{equation} where we have used that
$$\int_\R t^2(1+t^2)^{-2}dt=\pi/2,\text{ and }\int_\R(1+t^2)^{-2}dt=\pi/2.$$
To conclude so far, we obtain,
$$\partial_\eps|_{\eps=0} \|w_\eps\|_6^6 =c_0^6\frac {7\pi\sqrt\pi}{16\sqrt3}.$$
On the other hand,
\begin{equation}
\|w_0\|_6^6=c_0^6 \iint (1+t^2)^{-3/2} e^{-\frac {3|x|^2}{1+t^2}}dxdt =c_0^6\frac {\pi\sqrt\pi}{2\sqrt3}.
\end{equation}
Therefore we conclude that
\begin{equation}\label{eq-27}
\frac {\partial_\eps|_{\eps=0} \|w_\eps\|_6^6}{\|w_0\|_6^6}=\frac 78.
\end{equation}
We are left with computing $3\frac {\partial_\eps|_{\eps=0} \|g_\eps\|_2^2}{\|g_0\|_2^2}$:
\begin{equation}
\begin{split}
\partial_\eps|_{\eps=0} \|g_\eps\|_2^2& =\int_\R \left(-\frac 14 y^4+\frac 12y^2\right)e^{-y^2}dy\\
 &=-\frac 14 \times \frac {3\sqrt\pi}{4}+\frac 12 \times \frac {\sqrt\pi}{2}=\frac {\sqrt\pi}{16}.\\
\|g_0\|_2^2&=\int_\R e^{-y^2}dy =\sqrt{\pi}.
\end{split}
\end{equation}
Note that in the first inequality we have used that
$$
\int_\R y^4e^{-y^2} dy =\frac {3\sqrt\pi}{4},\text{ and }
\int_\R y^2 e^{-y^2} dy =\frac {\sqrt\pi}{2}.$$
So we have
\begin{equation}\label{eq-28}
3\frac {\partial_\eps|_{\eps=0} \|g_\eps\|_2^2}{\|g_0\|_2^2}=\frac 3{16}.
\end{equation}

Combining \eqref{eq-21}, \eqref{eq-27} and \eqref{eq-28}, we see that
\begin{equation}
\partial_\eps|_{\eps=0} \Psi(\eps )= \frac 78-\frac {3}{16}=\frac {11}{16}>0,
\end{equation} which establishes the claim in Lemma \ref{le-4}.

Then the following symmetry consideration completes the proof of Proposition \ref{prop-strict-inequ}: Let $f_\eps$ be defined in \eqref{eq-76}, and $\tilde{f}_\eps(x):=f_\eps(-x), \, F_\eps :=(f_\eps+\tilde{f}_\eps)/{\sqrt2}.$
Then $\|F_\eps\|_2=\|f_\eps\|_2$ and
\begin{equation}\label{eq-77}
\frac {\|F_\eps\sigma\ast F_\eps \sigma \ast F_\eps\sigma \|_2}{\|F_\eps\|_2^3}\ge (5/2)^{1/2} \frac {\|f_\eps\sigma\ast f_\eps\sigma\ast f_\eps\sigma \|_2}{\|f_\eps\|_2^3}.
\end{equation}

We focus on proving \eqref{eq-77}.
\begin{equation}\label{eq-78}
F_\eps \sigma\ast F_\eps\sigma \ast F_\eps\sigma =2^{-3/2} (f_\eps\sigma+\tilde{f}_\eps\sigma)\ast (f_\eps\sigma+\tilde{f}_\eps\sigma)\ast (f\sigma+\tilde{f}_\eps\sigma).
\end{equation}
Because of the identity \eqref{eq-81},
\begin{equation}\label{eq-79}
\begin{split}
\langle f_\eps\sigma \ast f_\eps\sigma \ast f_\eps\sigma,\,f_\eps\sigma \ast f_\eps\sigma \ast f_\eps\sigma \rangle &=
\langle f_\eps\sigma \ast f_\eps\sigma \ast \tilde{f}_\eps \sigma,\tilde{f}_\eps \sigma \ast f_\eps\sigma \ast f_\eps\sigma \rangle\\
&= \langle f_\eps\sigma \ast \tilde{f}_\eps \sigma \ast \tilde{f}\sigma,\tilde{f}_\eps \sigma \ast \tilde{f}_\eps\sigma \ast f_\eps\sigma \rangle\\
 &=\langle \tilde{f}_\eps \sigma \ast \tilde{f}_\eps \sigma \ast \tilde{f}_\eps \sigma,\tilde{f}_\eps \sigma \ast \tilde{f}_\eps \sigma \ast \tilde{f}_\eps \sigma \rangle.
 \end{split}
\end{equation}So by nonnegativity, we see that
\begin{equation}\label{eq-80}
\|F_\eps\sigma \ast F_\eps \sigma \ast F_\eps\sigma \|_2^2 \ge 2^{-3} (1+9+9+1) \|f_\eps \sigma \ast f_\eps \sigma \ast f_\eps \sigma \|_2^2 =
(5/2)\|f_\eps \sigma \ast f_\eps \sigma \ast f_\eps \sigma \|_2^2,
\end{equation}which yields \eqref{eq-77}.
\end{proof}

\end{document}